\documentclass[11pt]{amsart}
\usepackage{graphicx}
\usepackage{amssymb}

\usepackage{amssymb,amsmath,amsthm}
\usepackage[all]{xy}
\usepackage{url}
  \usepackage{color}
  \usepackage{xy}
 \xyoption{curve}
 \xyoption{color}
 \xyoption{line}
\xyoption{arc}
\usepackage{graphicx}
\usepackage[
           colorlinks=true,
            urlcolor=blue,       
            filecolor=blue,     
            linkcolor=blue,       
            citecolor=blue,         
            pdftitle= {Title},
            pdfauthor={Author},
             pdfsubject={Subject},
            pdfkeywords={Key}
            pagebackref,
            pdfpagemode=None,
            bookmarksopen=true]
            {hyperref}
\usepackage{hyperref}
\usepackage{color}
\usepackage[author=,english,status=draft]{fixme}
\usepackage{ifdraft}

\ifdraft{
\usepackage{fancyhdr}
\pagestyle{fancy}
\fancyhead{}
\fancyfoot{}
\fancyhead[CO,CE]{Draft --- \today}
\fancyhead[RO, LE] {\thepage}
}

\date{}

 \newtheorem{thm}{Theorem}[section]
 \newtheorem{prop}[thm]{Proposition}
 \newtheorem{lem}[thm]{Lemma}
 \newtheorem{corollary}[thm]{Corollary}

 \theoremstyle{definition}
 \newtheorem{example}[thm]{Example}
 \newtheorem{definition}[thm]{Definition}
 \newtheorem{remark}[thm]{Remark}
\numberwithin{equation}{section}

\newcommand{\bbA}{{\mathbb{A}}}

\newcommand{\bbG}{{\mathbb{G}}}

\newcommand{\bbP}{{\mathbb{P}}}

\newcommand{\bbZ}{{\mathbb{Z}}}

\newcommand{\bfF}{{\mathbf{F}}}

\newcommand{\frakf}{\mathfrak{f}}
\newcommand{\frake}{\mathfrak{e}}

\newcommand{\Gen}{\operatorname{Gen}}

\newcommand{\Cr}{\operatorname{Cr}}
\newcommand{\Sing}{\operatorname{Sing}}

\newcommand{\Bl}{\operatorname{Bl}}

\newcommand{\bp}{\operatorname{bp}}

\newcommand{\Pic}{\mathrm{Pic}}
\newcommand{\SL}{\mathrm{SL}}

\newcommand{\half}{\frac{1}{2}}
\newcommand{\da}{\dasharrow}

\newcommand{\la}{\langle}
\newcommand{\ra}{\rangle}

\newcommand{\calC}{\mathcal{C}}

\newcommand{\calO}{\mathcal{O}}
\newcommand{\calP}{\mathcal{P}}
\newcommand{\calD}{\mathcal{D}}
\newcommand{\calR}{\mathcal{R}}
\newcommand{\calH}{\mathcal{H}}
\newcommand{\calS}{\mathcal{S}}

\newcommand{\frakC}{\mathfrak{C}}

\newcommand{\dJ}{\mathrm{dJ}}

\newcommand{\norm}{\mathrm{norm}}

\newcommand{\beq}{\begin{equation}}
 
\newcommand{\eeq}{\end{equation}}

\title[Monoidal and submonoidal surfaces]{Monoidal and submonoidal surfaces, and Cremona 
transformations}
\author{Igor V. Dolgachev}
\address{Department of Mathematics,
University of Michigan, Ann Arbor, MI 48109}
\email{idolga@umich.edu}

\begin{document}
\dedicatory{
To the memory of Yuri I. Manin}

\begin{abstract} We study irreducible surfaces of 
degree $d$ in $\bbP^3$ that contain a line of multiplicity 
$d-1$ (monoidal surfaces)
 or $d-2$ (submonoidal surfaces). We relate them to congruences of lines and 
 Cremona transformations. Many of our results are not new and can be found in the classical 
literature, we give them modern proofs. In the last section, 
  we extend some of our results to hypersurfaces of arbitrary dimension. 
  We define two commuting Cremona involutions in the ambient space associated to 
  a linear subspace of multiplicity $d-2$ contained in the hypersurface. 
Both leave the hypersurface invariant, but one acts as the identity on the 
hypersurface. 
\end{abstract}

\maketitle

\tableofcontents
\section*{Introduction} 
 An irreducible reduced hypersurface $\Phi_d$ in $\bbP^n$ of degree $d > 1$ is called a 
 \emph{monoidal hypersurface} 
(resp. a \emph{submonoidal} hypersurface) if it contains a point of multiplicity $
d-1$ (resp. $d-2$). 

In the present paper, we study monoidal (submonoidal) surfaces $\Phi_d$  of degree $d\ge 3$ 
which contain a line $\Gamma$ of multiplicity $d-1$ (resp. $d-2$).

In the next section, we will discuss some canonical equations of monoidal surfaces. If $d = 3$, these equations 
give the classical canonical equations of two non-isomorphic ruled cubic surfaces which are not cubic cones. In the case $d = 4$,
they are ruled surfaces 
with a triple line. We find their canonical equations which simplify the equations which can be found in \cite{Edge}
or \cite{Salmon}. A monoidal surface is an example of a rational ruled surface in $\bbP^3$. We 
discuss possible normalizations of a monoidal surface, and also a construction
of a monoidal surface of degree $d$ as a projection of a minimal ruled surface embedded in $\bbP^{d+1}$.

In section 2, we study monoidal Cremona transformations defined by homaloidal 3-dimensional linear systems of monoidal 
surfaces with their singular line contained in the base locus. 
An example of a monoidal transformation is a quadro-cubic 
Cremona transformation defined by the linear system of cubic surfaces containing a fixed double line and three isolated points. 

 In Section 3, we introduce a condition of the 
non-degeneracy of a submonoidal surface. This condition is implicitly assumed in the classical 
literature. A non-degenerate submonoidal surface has only nodes as isolated singularities and its normalization 
is isomorphic to the proper transform in the blow-up of the singular line. The proper transform admits a structure of a conic bundle, and we study its possible singularities.
Using this, we reprove the classical result that a non-degenerate submonoidal surface admits at most $3d-4$ nodes as its isolated singular points. 
A quartic surface with a double line and eight nodes is the Pl\"ucker Complex Surface which arises in 
the theory of quadratic line complexes.

In section 4, we extend the classical construction of a birational model of quartic surface with a 
double line as the blow-up of nine points in the plane. We show that a submonoidal surface of degree $d$ 
is isomorphic to the image of the plane under a rational map defined by the linear system of plane curves 
of degree $d$ passing simply through $3d-4$ points $p_1,\ldots,p_{3d-4}$ and with multiplicity $d-2$ through an additional point $p_0$.
The line $\Gamma$ is the image of a unique curve of degree $d-1$ that passes through $p_1,\ldots,p_{3d-4}$ 
and has $p_0$ as a point of multiplicity $d-3$. We use the blow-up model to find special sections of the conic bundle on a submonodidal surface of degree $d$.
These are rational curves of degree $d-2$ that intersect $\Gamma$ at $d-3$ points. We show that there are 
$2^{3d-5}$ such curves and they come in pairs, the union of the curves in a pair is cut out by a hypersurface 
of degree $d-3$. We associate to a special section $S$ a congruence of lines $\calC_S$ of order 1 and class $d-2$. 
The surface $\calC_S$ is a ruled surface contained in the Grassmannian $G_1(\bbP^3)$, and we relate 
sections of the conic bundle on $\Phi_d$ to directrices of the ruled surface.  
 
In section 5, using the de Jonqui\`eres birational transformation of $\bbP^2$ associated with the hyperelliptic curve 
$\Sigma$, we construct a canonical Cremona involution of $\bbP^3$ that leaves the surface invariant and 
induces a de Jonqu\`eres involution of the blow-up plane model of the surface. 

One can extend some results of the paper to higher dimensional hypersurfaces in $\bbP^n$ of degree $d$ 
that contain a linear subspace $\Gamma$ with multiplicity $d-2$. We do not pursue this in this paper, 
however, in the last section, we briefly discuss such hypersurfaces and introduce the satellite polar variety canonically 
associated with them. We construct a Cremona transformations $\Theta_\Gamma$ of $\bbP^n$ that 
leaves the hypersurfaces invariant, 
it fixes pointwise 
$\Gamma$ and the satellite polar variety. As in the case of surfaces, there is 
another Cremona involution $\Theta_\Gamma'$ that fixes the 
hypersurface pointwise. In the case of cubic hypersurfaces with $0$-dimensional $\Gamma$, the
 transformations $\Theta_\Gamma$ 
were first studied by Yuri Ivanovich Manin \cite{Manin}. 

A large part of the paper should be considered as a modern exposition of the known facts which can be found in classical 
literature, for example, in monographs \cite{Jessop}, \cite{Salmon}, and \cite{Sturm2}. 
It would be hard to give the precise numerous references. 
 
\vskip5pt

The author thanks J\'er\'emy Blanc, Ivan Cheltsov and Alessandro Verra for helpful comments on the first draft of the paper.

\section{Monoidal surfaces}
\subsection{Canonical equations} Let $\Phi_d$ be a monoidal surface containing a line $\Gamma$ 
with multiplicity 
$d-1$.
We identify the pencil of planes through $\Gamma$ with the dual line $\Gamma^\perp$ in the dual $\bbP^3$. 
 If $\Phi_d$ contains a singular point $x_0$ outside $\Gamma$, then every line through $x_0$ that 
 intersects $\Gamma$ is contained in the surface, and hence $\Phi_d$ is a cone. We will exclude this case.
 
  Intersecting $\Phi_d$ with planes from the pencil 
$\Gamma^\perp$ shows that $\Phi_d$ is a ruled surface  
with the ruling defined by the residual lines.

We choose coordinates such that $\Gamma = V(x_0,x_1)$ and write the equation of $\Phi_d$ as 
\beq\label{eqn1}
F_d = A_d(x_0,x_1)+x_2B_{d-1}(x_0,x_1)+x_3C_{d-1}(x_0,x_1) = 0,
\eeq
where $A_d,B_{d-1}$ and $C_{d-1}$ are binary forms of degrees indicated by the subscripts.
Since $\Phi_d$ is irreducible, the binary forms $A_d,B_{d-1},C_{d-1}$ are mutually coprime.
 If $B_{d-1}$ and $C_{d-1}$ are proportional, say $B_{d-1} = t_0 C_{d-1}$, then, replacing  $x_2$ with 
 $x_2+t_0 x_3$, we obtain that $\Phi_d$ is a cone with vertex $[0:0:0:1]$.  We excluded this case.
 
 It follows from the adjunction formula \eqref{adjunction} that any automorphism of a monoidal or submonoidal surface 
 comes from a projective automorphism. Since the group of projective automorphisms of $\bbP^3$ that leaves invariant 
 a line is of dimension $11$,  
the dimension of the moduli space of monoidal surfaces of degree $d\ge 4$ is equal to $3d-11$. 
 
\begin{example}\label{ex:cubic1} Assume $d = 3$. If $B_2$ and $C_{2}$ are coprime binary quadrics, 
we can find linear forms $L_1$ and $L_2$ such that $A_3 = L_1B_2+L_2C_2$. This eliminates $A_d$ and, after reducing 
$C_2$ and $C_3$ to a sum of squares, we reduce
the equation to the form
\beq\label{eqncubic1}
x_0^2x_2+x_1^2x_3 = 0.
\eeq
If $B_2$ and $C_2$ share a root, we may assume that $x_0|B_2$ and $x_0|C_2$. We get 
$$x_1^3+x_0L_1(x_0,x_1)x_2+x_0L_2(x_0,x_1)x_3 = 0,$$
and we can further assume, after a linear change of $x_2$ and $x_3$, that $L_1 = x_1$ and $L_2 = x_0$.
We arrive at the canonical equation
\beq\label{eqncubic2}
x_1^3+x_0x_1x_2+x_0^2x_3 = 0.
\eeq
Equations \eqref{eqncubic1} and \eqref{eqncubic1} are well-known canonical equations of 
non-conical cubic ruled surfaces (see \cite[Theorem 9.2.1]{CAG}).
\end{example}

 We could use a further reduction of canonical equations by applying the theory of invariants of 
pencils of binary forms. A linear change of variables $(x_1,x_2)$ changes a basis in the pencil of binary forms 
of degree $d-1$ spanned by $B_{d-1}$ and $C_{d-1}$. If these forms are coprime, they span  
a base-point-free pencil in the space $V(d-1)$ of binary forms of degree $d-1$. It 
can be considered as a point in the Grassmannian $G(2,V(d-1))$. The group $\SL(2)$ acts faithfully on $G(2,V(d-1))$ 
via its natural linear 
representation in $V(d-1)$. It is known that the orbit of the pencil depends only on the ramification scheme of the map 
$\bbP^1\to \bbP^1$ of degree $d-1$ defined by the pencil \cite{DeopurkarPatel}. A general orbit corresponds to a map 
with $2d-4$ simple ramification points, such orbits are parameterized by a variety of dimension $2(d-2)-3 = 2d-7$. 
After we fix an orbit, we can change the variables 
\beq\label{changeover}
(x_1,x_2)\mapsto (cx_1+l_1(x_0,x_1),x_2+l_2(x_0,x_1),
\eeq
where $l_1,l_2$ are 
linear forms. This leaves us with $(2d-7)+(d+1-5) = 3d-11$ parameters that agrees with the previous count of constants.

 \begin{example}\label{ex:quartics1} Assume $d = 4$. 
 A degree $3$ map $f:\bbP^1\to \bbP^1$ may have possible ramification schemes: $(2,2,2,2), (2,2,3),$ and $(3,3)$. 
The orbits of pencils of cubic binary forms with four simple ramification points depend on one parameter.  
 They can be represented by pencils of polar cubics of 
a general quartic binary form. A canonical equation of a general binary quartic form 
is $x_0^4+2\lambda x_0^2x_1^2+x_1^4$. 
Thus we may assume that $B_3 = x_0^3+\lambda x_0x_1^2$ and $C_3 = x_1^3+\lambda x_0^2x_1$. 
The pencil defines a map $\bbP^1\to \bbP^1$ with four simple ramification points if and only 
if $\lambda \ne 0,\pm 1,\pm 3$. In this case, we can use \eqref{changeover} to transform 
$A_4(x_0,x_1)=\sum_{i=0}^4a_ix_0^{4-i}x_1^i$ to $x_0^2x_1^2$ or to $0$ 
(if $a_2 = \lambda (a_0+a_4)$). This gives canonical equations of a general monoidal quartic surface
\begin{eqnarray}\label{(i)}
&&x_0^2x_1^2+(x_0^3+\lambda x_0x_1^2)x_2+(x_1^3+\lambda x_0^2x_1)x_3 = 0,\quad \lambda^2 \ne 0,1,9.\\
&&(x_0^3+\lambda x_0x_1^2)x_2+(x_1^3+\lambda x_0^2x_1)x_3 = 0, \quad \lambda^2 \ne 0,1,9.\label{(ii)}
\end{eqnarray}
If we take $\lambda = \pm 3$ (replacing  $x_0$ with $ix_0$, we may assume $\lambda = 3$),
the pencil has the ramification scheme $(2,2,3)$. 
 There is a unique orbit with this ramification
scheme, so we may take the pencil of polars of $x_0^4+6x_0^2x_1^2+x_1^4$ as its representative.  
This gives the normal forms
\begin{eqnarray}\label{(iii)}
&&x_0^2x_1^2+(x_0^3+3x_0x_1^2)x_2+(x_1^3+3x_0^2x_1)x_3 = 0,\\
&&(x_0^3+3x_0x_1^2)x_2+(x_1^3+3x_0^2x_1)x_3 = 0.\label{(iv)}
 \end{eqnarray}
If $\lambda = \pm 1$, we obtain a pencil with two base points (may be equal)) that gives canonical equation
\begin{eqnarray}\label{(v)}
&&x_0^2x_1^2+(x_0^2+x_1^2)(x_0x_2+x_1x_3) = 0,\\
&&x_0^2x_1^2+(x_0+x_1)^2(x_0x_2+x_1x_3) = 0,\label{(vi)}
\end{eqnarray}
There is one more case of pencils with one base point. It gives canonical equation
\beq\label{(vii)}
x_0^2x_1^2+(x_0+x_1)(x_0^2x_2+x_1^2x_3) = 0.
 \eeq

All normal forms reveal a non-trivial projective symmetry $(x_0,x_1,x_2,x_3)\mapsto (x_1,x_0,x_3,x_2)$.
Our equations agree with Salmon's equations from \cite[Chapter XVI, Art. 546--549]{Salmon}. Equations \eqref{(i)} and 
\eqref{(ii)} (resp. equations \eqref{(iii)} and \eqref{(iv)}) are  Salmon's type I  (resp, II). 
Equation \eqref{(vii)} (resp. \eqref{(v)}, resp. \eqref{(vi)}) is Salmon's type III (resp. IV, resp. V).

Salmon also finds the equations of the dual surfaces, some of them are self-dual. 

\end{example}

\subsection{Monoidal surfaces and rational ruled surfaces in $\bbP^3$}
A ruled surface $F$ in $\bbP^N$ is the image in $\bbP^N$ of the tautological 
$\bbP^1$-bundle over an irreducible 
curve $\Gen(F)$ in the Grassmannian $G_1(\bbP^N)$ of lines in $\bbP^N$. The curve $\Gen(F)$ 
is the 
\emph{generatrix} of $F$. The lines of the ruling are \emph{generators} of $F$. 
The pull-back of the $\bbP^1$-bundle to the normalization $\Gen(F)^{\norm}$ is 
a minimal ruled surface over $\Gen(F)^{\norm}$. When $F$ is rational, and only those 
we will encounter in this paper, $\Gen(F)^{\norm}\cong \bbP^1$ and 
the minimal ruled surface is isomorphic to $\bfF_n: = \bbP(\calO_{\bbP^1}\oplus 
\calO_{\bbP^1}(-n))$ for some $n\ge 0$. The map $\bfF_n\to F$ is the normalization map unless 
$F$ is a cone.  We assumed that $F$ is not a cone.
 
Another invariant of a ruled surface besides its normalization  
is the smallest degree of its \emph{directrix}, an irreducible curve on $F$ that intersects 
each generator at one point and is not contained in $\Sing(F)$. If $F$ is smooth, the directrix $D_{\min}$ of 
smallest degree coincides with the image of the exceptional section 
$\frake$ of $\bfF_n$ (or of any generator from another ruling on $\bfF_0$). 

A non-degenerate rational ruled surface $F$ of degree $d$ in $\bbP^{d+1}$ is an example of a 
non-degenerate surface of minimal degree. In particular, it must be a normal surface. Conversely, 
a normal ruled surface is isomorphic to a non-degenerate surface of minimal degree.

Take $d-1$ general points $x_1,\ldots,x_{d-1}$ on $F$ that span 
a projective subspace $L^{d-2}$ of dimension $d-2$. The projection to $\bbP^2$ with the center at $L^{d-2}$ defines 
a rational map $F\da \bbP^2$. It can be regularized by passing to 
the blow-up $\Bl_\Sigma(F)$ of the set of points $\Sigma = \{x_1,\ldots,x_{d-1}\}$. The birational morphism 
$\Bl_\Sigma(F)\to \bbP^2$ blows down the proper transforms $g_i'$ of generators $g_i$ of $F$ passing through $x_i$ 
to points $p_i$ in the plane. The span of $L^{n-2}$ and a generator $g$ skew to it spans a hyperplane that cuts out $F$ along $g$ and a curve 
$C_{d-1}(g)$ of degree $d-1$. When we let $g$ vary in a pencil, we obtain a pencil of hyperplanes. It corresponds 
to the pencil of lines in the plane passing through a point $O$. Taking $g = g_i$, we get a hyperplane section that consists
of $g_i$ and a curve of degree $d-1$ containing $x_1,\ldots,x_{d-1}$. This means that $O$ is a 
$(d-1)$-multiple point 
of the projections of hyperplane sections. Also, it shows $F$ admits a rational parameterization 
$\phi:\bbP^2\da F$ given by the complete linear system 
\beq\label{beltrametti}
|L| = C^d(p_0^{d-1},p_1,\ldots,p_{d-1})
\eeq
of curves of degree $d$ passing through points $p_1,\ldots,p_{d-1}$ and 
containing a point $p_0$ with multiplicity $\ge d-1$ (see \cite[10.3.6]{Beltrametti}). 
The map $\phi$ regularizes on the blow-up $X = \Bl_{p_0,p_1,\ldots,p_{d-1}}(\bbP^2)$. 
In the standard geometric 
basis $(e_0,e_1,\ldots,e_d)$ of $X$, where $e_d$ is the divisor class of the exceptional curve 
$E_d$ over $p_0$, the proper transform of $|L|$ is given by 
$$|M| = |de_0-e_1-\cdots-e_{d-1}-(d-1)e_d|.$$

 The pencil $|e_0-e_d|$ is mapped 
to the ruling of $F$. A directrix $D_m$ of degree $m$ is the image of a curve $C$ on $X$ 
such that $C\cdot e_i \le 1, i= 1,\ldots, d-1, C\cdot (e_0-e_d) = 1$, and $C\cdot M = m$. It implies that 
$$C \in  |(m+\#I-d+1)e_0-\sum_{i\in I}e_i-(m+\#I-d)e_{d}|,$$
where $I = \{i\in \{1,\ldots,d-1\}, x_i\not\in D_m \}$. 

For example, the image of general line in 
the plane is a directrix 
of degree $d$ that contains all the points $x_i$. We have 
$$C^2 = 2m-2d+1+\#I, \quad -C\cdot K_X = 2m-2d+3+\#I,$$
so we expect that 
$$\dim |C| =  2m+\#I-2d+2.$$
It suggests that the smallest degree of a directrix is equal to
$k = \lceil \frac{d-1}{2}\rceil$, and in this case $\#I = d-1$. It  represents a curve 
$C^k(p_0^{k-1},p_1,\ldots,p_{d-1})$ in the plane. The curve is unique if $d$ is even, and vary in a pencil if $d$ is odd.

In fact, we have the following (see \cite[10.4.6]{CAG}).

\begin{lem}\label{directrix} Let $F\cong \bfF_e\hookrightarrow \bbP^n$ be an embedding of a minimal rational 
ruled surface $\bfF_e$ of degree $n-1$ by a complete linear system
$|a\frakf+\frake|$, where $\frake^2 = -e, 2a-e= n-1$. Then the  degree $\deg(D)$ of a directrix $D$ on $X$ satisfies 
$$\deg(D)\ge a-e = \half (e+\deg(F))-e,$$ 
and the equality takes place if and only if $D$ is the image of the exceptional section from $|\frake|$.
\end{lem} 

Our ruled surface $\Phi_d$ is the image of some $\bfF_e$ under a linear projection to $\bbP^3$. 
The projection of 
a directrix is a directrix on $\Phi_d$ of the same degree. Thus, the smallest degree $d_{\min}$ of a directrix on 
$\Phi_d$ is equal to $\half (e+d)-e\le \lceil \frac{d-1}{2}\rceil$. The equality takes place if and only if 
$(d,e) = (2m+1,0)$, or $(d,e) = (2m,1))$.

 So we expect that the normalization $F^{\norm}$ of a general $F$ is 
isomorphic to $\bfF_1$ (resp. $\bfF_0$) if $d$ is even (resp. $d$ is odd). For example, if $A_d = 0$ in its equation \eqref{canequation}, then the line $V(x_2,x_3)$ is its
 directrix 
of degree $1$.

If $F$ is not general, then a directrix $D$ of degree $m$ smaller than 
$\lceil \frac{d}{2}\rceil$ necessarily passes through some points $x_i$, and it has negative self-intersection. 
Since it coincides with the exceptional section of $F^{\norm}$, we get 
\beq\label{fn}
F\cong \bfF_{d-2m}.
\eeq

Note that the blow-up of $\bbP^2$ at the point $p_0$ is a minimal ruled surface isomorphic to $\bfF_1$.
The surface $X = \Bl_{p_0,\ldots,p_{d-1}}(\bbP^2)$ is obtained from this surface by blowing up $d-1$ points not lying 
on the exceptional section of $\bfF_1$. The map $X\to F$ blows down the proper transforms of the 
generators passing through these points. We see that $F$ is obtained from $\bfF_1$ by a composition of $d-1$ 
elementary birational transformations $\textrm{elm}_{x_i}$ (see \cite[7.4.2]{CAG}).

The normalization $\Phi_d^{\norm}$ of $\Phi_d$ is isomorphic to the proper transform of $\Phi_d$ in 
the blow-up $\sigma:\Bl_{\Gamma}(\bbP^3)\to \bbP^3$. It is well-known that 
\beq\label{blowup}
\Bl_{\Gamma}(\bbP^3)\cong \bbP(\calO_{\bbP^1}^{\oplus 2}\oplus\calO_{\bbP^1}(1)).
\eeq
The projection $\pi:\bbP\to \bbP^1$ is given by the pencil of planes through $\Gamma$, so we may identify the base of the projective 
$\bbP^2$-bundle $\bbP$ with $\Gamma^\perp$. 

We have 
\beq\label{pic}
\Pic(\bbP) = \bbZ H +\bbZ P,
\eeq
where $H  = c_1(\calO_\bbP(1)) = \sigma^*c_1(\calO_{\bbP^3}(1))$ and $P = \pi^*c_1(\calO_{\bbP^1}(1))$. The exceptional divisor 
$E$ of the blow-up is linearly equivalent to $H-P$ and 
\beq\label{classX}
\Phi_d^{\norm} \sim dH-(d-1)E \sim H+(d-1)P.
\eeq 

We know from \eqref{fn} that 
$$\Phi_d^{\norm}\cong \bfF_n,\quad n = d-2m,$$
where $m$ is the smallest degree of a directrix of $\Phi_d$. 

We use that 
$$\bbP = (\bbA^3\setminus \{0\})\times (\bbA^2\setminus \{0\})/\bbG_m^2,$$
where the action is defined by 
$$(\lambda,\mu):((y_0,y_1,y_2)),(t_0,t_1))\mapsto ((\lambda^{-1} y_0,\mu y_1,\mu y_2),(\lambda t_0,\lambda t_1)).$$
with the projections 
$$\sigma:\bbP\to \bbP^3, \ ([y_0,y_1,y_2],[t_0,t_1])\mapsto [t_0y_0,t_1y_0,y_1,y_2]$$
and the projection 
$$\sigma:\bbP\to \bbP^1 = \Gamma^\perp, \ ([y_0,y_1,y_2),(t_0,t_1)]\mapsto [t_0,t_1].$$

The proper transform of $\Phi_d$ in $\bbP$ is now given by equation
\beq\label{canequation}
A_d(t_0,t_1)y_0+B_{d-1}(t_0,t_1)y_1+C_{d-1}(t_0,t_1)y_2 = 0.
\eeq

It can be also considered as the equation of a residual line. It shows that each point 
on $\Gamma$ is contained in $d-1$ residual lines.

The curve $\Sigma$ contains a unique irreducible component $\Sigma_0$ that defines a section of 
$\sigma:\Gamma^\perp\to \bbP$. All other components are fibers of the projection 
$\sigma|E:E\to \Gamma^\perp$.

 The equation of the intersection $\Sigma:=\tilde{\Phi}_d\cap E$ becomes
$$B_{d-1}(t_0,t_1)y_1+C_{d-1}(t_0,t_1)y_2 = 0.$$
It exhibits $\Sigma$ as a curve in $E\cong \Gamma\times \Gamma^\perp$ of bidegree $(1,d-1)$. 
The projection $\Sigma\to \Gamma$ is a finite map of degree $d-1$. 
 These are the points $x\in \Gamma$ where a plane in $\Gamma^\perp$ is tangent to $\Phi_d$ along 
a line passing through $x$. If $\Sigma$ is smooth, the branch points are called the \emph{pinch points}. We expect 
$2(d-2)$ pinch points.

Let $\Phi_d^{\norm}\cong \bfF_{d-2m}\hookrightarrow \bbP^{d+1}$ be the embedding of 
$\Phi_d^{\norm}$ by the complete linear system $|(d-m)\frakf+\frake|$.


The surface $\Phi_d$ is a projection of $\Phi_d^{\norm}\subset \bbP^{d+1}$ to $\bbP^3$ from a subspace 
$L^{d-3}$ of codimension $4$. The image of the curve $\Sigma$ spans a subspace $M^{d-1}$ of codimension $2$. 
The projection map $p:\Sigma\to \Gamma$ is of degree $d-1$, hence $L^{d-3}$ must be contained in $M^{d-1}$. 
The projection $p$ is defined by the pencil of hyperplanes in $M^{d-1}$ containing $L^{d-3}$. 

The image of the exceptional section is a directrix $D_{\min}$ of degree $m$ that spans a linear subspace $M^{m-1}$ of dimension 
$m-1$ disjoint from the center of the projection. Since $\frake\cdot [\Sigma] = m-1$, the 
directrix $D_{\min}$ intersects $\Gamma$ at $m-1$ points.

\begin{example}\label{ex:cubic2} Assume $d = 3$ and keep the notation from Example \ref{ex:cubic1}. 
It follows from the previous 
formulas that the smallest degree of a directrix on $\Phi_3$ is equal to 1. In fact, this can be easily derived from the 
equations. If $A_3= 0$, then $V(x_2,x_3)$ is a line directrix. 
If $A_3\ne 0$, the line $V(x_1,x_3)$ is a line directrix. 
Thus $\Phi_d^{\norm}\cong \bfF_1$ embedded in $\bbP^4$ by the linear system $|2\frakf+\frake|$. It is a
smooth cubic ruled surface. The projection from a point in $\bbP^4\setminus \Phi_d^{\norm}$ maps 
it to $\Phi_d$. There are two possibilities:
the curve $\Sigma\in |\frakf+\frake|$ is irreducible, or $\Sigma$ is the union of the exceptional section 
identified with $\frake$ and a fiber $F\in |\frakf|$. In the first case, we get a surface with canonical equation \eqref{eqncubic1}. The 
double curve is a non-exceptional section. The image of the exceptional section is the line $x_2=x_3 = 0$. 
In the second case, the pre-image of $\Gamma$ consists of the exceptional section and a line from the ruling. 

\end{example}

\begin{example} Assume $d = 4$.
Then the minimal degree $m$ of a directrix is equal to $2$ (this is the general case) or $m = 1$.
In the first case, 
$|2\frakf_1+\frakf_2|$ maps $\Phi_d^{\norm}\cong \bfF_0$ to a smooth quartic ruled surface in $\bbP^5$, and 
in the second case, $|3\frakf+\frake|$
 embeds $\Phi_d^{\norm}\cong \bfF_2$ in $\bbP^5$, also as a smooth quartic surface. The center of the 
 projection is a line disjoint from $\tilde{\Phi}_d$. In the classification of 
 quartic ruled surfaces in $\bbP^3$ 
 (see \cite[10.4.4]{CAG}, \cite{Edge}), the first case 
corresponds to Types II (A) and III (A) in Edge's notation. 
His equations of the surfaces can be simplified and correspond to 
 our canonical equations \eqref{(i)},\eqref{(iii)} for Type II (A),  and \eqref{(vii)} for Type III (A). 
 
 The second case corresponds to Type IV (A) and V (B). Edge's equations 
are our equations \eqref{(ii)}, \eqref{(ii)} for Type IV (A) and \eqref{(v)},\eqref{(vi)} for Type V (B). 
 The double curve is irreducible if the surface is of Type IV (A) in Edge's notation.
 \end{example}

\section{Monoidal Cremona transformations}
 According to a modern terminology, a \emph{monoidal birational 
 transformation} (or a $\sigma$-process) is the blow-up of a smooth subvariety of a projective space 
 \cite[Chapter IV, \S 3]{Hartshorne}. The following is the classical definition of a monoidal 
 trasnformation \cite{Hudson}.

 \begin{definition} A birational transformation $T:\bbP^n\da \bbP^n$ is called a \emph{monoidal transformation} if 
 it is given by a linear system  whose general member is a monoidal hypersurface of degree $d$ whose singular 
 locus $\Gamma$ of multiplicity 
 $d-1$ is contained in the base locus.
 \end{definition}
 
 The two definitions are obviously related, in order to resolve the indeterminacy points of a monoidal transformation 
 one has to blow up its base locus.
 
 Here we will discuss a particular case when $n = 3$ and monoidal surfaces are ruled surfaces of degree $d$ 
 with a $(d-1)$-multiple line.

 \begin{lem}  Let $\Phi_d$ and $\Phi_d'$ be two monoidal surfaces of degree $d$ 
 containing a line $\Gamma$ with multiplicity $d-1$.  Then
 $$\Phi_d\cap \Phi_d' = (d-1)^2\Gamma+C_{2d-1},$$
 where $C_{2d-1}$ is a curve of degree $d^2-(d-1)^2 = 2d-1$. 
 Assume that $C_{2d-1}$ intersects transversally $\Gamma$. Then the number of intersection points 
 is equal to $2d-2$.
 \end{lem}
 
 \begin{proof} Take a general plane $\Pi$ in $\Gamma^\perp$. Its residual intersections with 
$\Phi_d$ and $\Phi_d'$ are lines. They intersect at one point. This shows that $\Pi$ intersects $C_{2d-1}$ 
along $C_{2d-1}\cap \Gamma$ and one point. Hence $\#C_{2d-1}\cap \Gamma = 2d-1-1 = 2(d-1)$.
\end{proof}

 The linear system of monoidal surfaces of degree $d$ with a line $\Gamma$ of multiplicity $d-1$ is of dimension $(d+1)+2d-1 = 3d$.
Choose $\alpha$ general lines 
 $\ell_1,\ldots,\ell_{\alpha}$ intersecting 
 $\Gamma$ and $\beta$ general points $P_1,\ldots,P_{\beta}$ in $\bbP^3$. Since passing through $\ell_i$ imposes 2 
 conditions on monoidal surfaces, we 
 obtain that the dimension of the linear system 
 $$|L| = |\calO_{\bbP^3}(d)-(d-1)\Gamma-\sum_{i=1}^{\alpha}\ell_i-\sum_{i=1}^{\beta}P_i|$$
 is equal to $3d-2\alpha-\beta$. It is a web (i.e $\dim |L| = 3$) if
 \beq\label{cond1}
 2\alpha+\beta = 3(d-1).
 \eeq
 Suppose that the web $|L|$ is homaloidal, that is, it defines a birational map $T:\bbP^3\da \bbP^3$.   
 Its restriction to a general plane $\Pi$ is given by 
 the monoidal linear system of curves of degree $d$ with $(d-1)$-multiple point $p_0 = \Pi\cap \Gamma$ 
 and $\alpha$ simple points 
 $p_i = \ell_i\cap \Pi$.  
 The image of $\Pi$ under the restriction of $T$ to $\Pi$ is a monoidal surface of degree 
 $$d' = 2d-1-\alpha.$$ 
Varying $\Pi$, we obtain a homaloidal web of monoidal surfaces of degree $d'$. 
 It defines the inverse birational transformations 
 $T^{-1}$.
 
 Replacing $T$ with $T^{-1}$, we get the following equalities 
\beq\label{cond2}
2\beta+\alpha = 3(d'-1), \quad \beta = 2d'-d-1
\eeq
 (see \cite[p.216]{Hudson}).
 The image of each line in the plane through $p_0$ is a line. The union of lines 
 $\la p_i,p_0\ra$ and any line through $p_0$ form a subpencil of cones with vertex at $p_0$ in the linear system. 
 This shows that the image of the point $p_0$ is a line $\Gamma'$ taken with multiplicity $d-1$

Consider the projective space of monoidal surfaces of degree $d-1$ with a line $\Gamma$ of multiplicity $d-2$. 
The dimension of this space is equal to $3d-3$. Using \eqref{cond1}, we see that there exists a unique surface $J_{d-1}$ from this space that 
contains $\alpha$ lines $\ell_i$ and $\beta$ points $P_i$. The union $J_{d-1}+\Pi$, 
where $\Pi\in \Gamma^\perp$, belongs to $|L|$. This shows that the image of the surface $J_{d-1}$ under $T$ is a line $\Gamma'$.
In other words, $J_{d-1}$ is an exceptional surface of $T$. Other exceptional surfaces are the planes 
$\Pi_i=  \la \Gamma,\ell_i\ra, i = 1,\ldots,\alpha,$ and $\Pi_i' = \la \Gamma,P_i\ra, i = 1,\ldots,\beta$. The 
degree of the jacobian of the affine transformation $\bbA^4\to \bbA^4$ defined by $T$ is equal to $4(d-1)$. 
Each plane $\Pi_i$ is blown down to a point, hence it must be taken with multiplicity $2$ in the jacobian. We get 
$4(d-1)-(d-1) -2\alpha-\beta = 3(d-1)-2\alpha-\beta = 0$. This shows that all exceptional surfaces are accounted for.

The line $\Gamma'$ is a common singular line of monoidal surfaces of degree $d'$ from 
the homaloidal web defining the inverse transformation $T^{-1}$. The exceptional surface $J_{d'-1}'$ of $T^{-1}$ is equal 
to the image of the exceptional divisor $E$ under the composition $\Bl_\Gamma(\bbP^3)\to \bbP^3\overset{T}{\da} \bbP^3$.

\begin{example}\label{2-3cremona} Let us take $d = 2, \alpha = 0, \beta = 3$. This gives $d' = 3$.
The transformation $T$ is an example of a quadro-cubic transformation. Its base locus consists of a line $\Gamma$ and 3 
isolated points $P_i$ outside $\Gamma$. The inverse transformation $T^{-1}$ is given by 
monoidal cubic surfaces with singular line $\Gamma'$ that
contains three skew lines $\ell_1,\ell_2,\ell_3$ intersecting $\Gamma'$. The exceptional surfaces of $T$ are the three planes $\la \Gamma,P_i\ra$ and 
the plane $\la P_1,P_2,P_3\ra$. The exceptional surfaces of $T^{-1}$ are the three planes $\la \Gamma',\ell_i\ra$ 
and the unique quadric surface containing the lines $\Gamma',\ell_1,\ell_2,\ell_3$.
 
If  $d = 3$, we can solve for possible $\alpha,\beta$ and obtain Cremona transformations $T$ 
 with $T^{-1}$ of possible algebraic degrees $d' = 2,3,4,5$ (see \cite{Deserti}, \cite{Hudson}).
 \end{example}

\begin{example} Let us take $\alpha = \beta = d-1$. In this case, the algebraic degree of $T^{-1}$ is also equal to $d$.
After we compose the transformation with a projective transformation, we will be able to identify the source and the target $\bbP^3$ and assume
that $\Gamma = \Gamma'$. The restriction of $T$ to each non-exceptional plane $\Pi\in \Gamma^\perp$ is a projective transformation 
that sends $\Pi$ to a plane $\Pi'\in \Gamma^\perp$. One can compose $T$ with another projective transformation that leaves 
$\Gamma$ invariant to obtain a Cremona involution \cite{Montesano}.
\end{example}

\section{Non-degenerate submonoidal surfaces}
\subsection{Conic bundle structure}
Let $\Phi_d$ be a submonoida surface of degree $d$ that contains a 
line $\Gamma$ of multiplicity $d-2$. 
For any general point 
$x\in \bbP^3$, the plane $\Pi  = \la x,\Gamma\ra$ intersects $\Phi_d$ along $\Gamma$ taken with 
multiplicity $d-2$ and a conic. We assume, as before, that $\Gamma$ is given by equations $x_0 = x_1 = 0$. Then 
the equation of  $\Phi_d$ is of the form

\beq\label{eq:Phi}
\begin{split}
&A_{d}(x_0,x_1)+2B_{d-1}(x_0,x_1)x_2+
2C_{d-1}(x_0,x_1)x_3\\
&+D_{d-2}(x_0,x_1)x_2^2+2E_{d-2}(x_0,x_1)x_2x_3
+F_{d-2}(x_0,x_1)x_3^2= 0.
\end{split}
\eeq
We will assume that $\Gamma$ does not contain points of multiplicity $d-1$.  
Otherwise, the surface is monoidal. The equation depends on $(d+1)+2d+3(d-1)-1 = 6d-3$ parameters. The group of projective automorphisms 
of $\bbP^3$ that fix a line is of dimension $11$. Thus we expect that the dimension of the moduli space 
of submonoidal surfaces is equal to $6d-14$.

 As in the previous section, we consider the projective 
bundle 
$\pi:\bbP = \Bl_\Gamma(\bbP^3)\to \Gamma^\perp$ and denote by $\Phi_d^{\bp}$ the proper transform of 
$\Phi_d$ under the blow-up map $\sigma:\Bl_\Gamma(\bbP^3)\to \bbP^3$. 

As in the case of monoidal surfaces, we use the coordinates $(y_0,y_1,y_2,t_0,t_1)$ in 
$\bbP$ to obtain the equation of $\Phi_d^{\bp}$ 
\beq\label{eq:conicbundle}
\begin{split}
&A_{d}(t_0,t_1)y_0^2+2y_0(B_{d-1}(t_0,t_1)y_1+
C_{d-1}(t_0,t_1)y_2)\\
&+D_{d-2}(t_0,t_1)y_1^2+2E_{d-2}(t_0,t_1)y_1y_2
+F_{d-2}(t_0,t_1)y_2^2 = 0.
\end{split}
\eeq
The line $\Gamma^\perp$ has coordinates $(t_0,t_1)$. The fibers of the projection $\pi$ are isomorphic 
to the residual conics 
$C_{[t_0,t_1]}$ in the plane $\Pi_{[t_0,t_1]}\in \Gamma^\perp$ (the common zeros of $D_{d-2},E_{d-2},F_{d-2}$ correspond to conics that contain 
$\Gamma$ as an irreducible component). 

Computing the discriminant of \eqref{eq:conicbundle}, we find that it is equal to 
\beq\label{discriminant} P(t_0,t_1) = \det\begin{pmatrix}A_d&B_{d-1}&C_{d-1}\\
B_{d-1}&D_{d-2}&E_{d-2}\\
C_{d-1}&E_{d-2}&F_{d-2}\end{pmatrix}.
\eeq
It is a binary form of degree $3d-4$. A point $[a,b]$ is a zero of $P(t_0,t_1)$ if and only if 
the fiber of the conic bundle is singular. For a general $\Phi_d$, we expect that the polynomial is reduced, and 
we get $3d-4$ singular fibers. We will clarify this later.

Let 
$\calD_1 = V(P(t_0,t_1))$, considered as a closed subscheme of $\bbP^1 = \Gamma^\perp$. We call it 
the \emph{discriminant} of the conic bundle.

 The intersection 
of $\Phi_d^{\bp}$ with the exceptional divisor $E\cong \Gamma\times \Gamma^\perp \cong \bbP^1\times 
\bbP^1$ is a 
curve $\Sigma$ of bidegree $(d-2,2)$ given by equation
\beq\label{eq:delta}
D_{d-2}(t_0,t_1)y_1^2+2E_{d-2}(t_0,t_1)y_1y_2
+F_{d-2}(t_0,t_1)y_2^2 = 0. 
\eeq
The projections $\sigma:E\to\Gamma$ and $\pi:E\to\Gamma^\perp$ define two projections 
$$p:\Sigma\to \Gamma, \quad q:\Sigma\to \Gamma^\perp$$
of degrees $d-2$ and $2$, respectively.

If $\Sigma$ is smooth, then its genus is equal to $d-3$, and, applying the Hurwitz formula, we expect that the cover 
$p:\Sigma\to \Gamma$ has $4(d-3)$ branch points. They are called the \emph{pinch points}. 
They are the points $x\in \Gamma$ at which a general plane section passing through $x$ has less than 
 $d-2$ local branches at $x$.

Let 
$$R(t_0,t_1): = E_{d-2}(t_0,t_1)^2-D_{d-2}(t_0,t_1)F_{d-2}(t_0,t_1).$$
It is a binary form of degree $2(d-2)$. A point $[a,b]$ is a root of $R(t_0,t_1)$ if and only if it is a branch point of 
the projection $q:\Sigma\to \Gamma^\perp$. Equivalently, $R(a,b) = 0$ if and only 
if the residual conic $C_t$ in the 
plane $\Pi_t$ intersects $\Gamma$ at one point.


Let $\calD_2 = V(R(t_0,t_1))$ considered as a closed subscheme of $\Gamma^\perp$. We call it 
the \emph{small discriminant} of the conic bundle.

\begin{lem}\label{lem:singDelta} A point $x = ([a,b],[c,d])\in \Sigma$ is a singular point of $\Sigma$ if and only if 
$[a,b]$ is a multiple root of $R(t_0,t_1)$. 
\end{lem}

\begin{proof} Here, and in the sequel, we use that a quadratic form over a discrete valuation ring of 
characteristic $\ne 2$ 
can be reduced to a 
sum of squares \cite[Satz (15.1)]{Kneser}. By localizing, we may assume that 
the quadratic form \eqref{eq:delta} over $A = \calO_{\bbP^1,[a,b]}$ with local parameter $t$ is given by 
$t^nx^2+t^{n+k}y^2$. If $n > 0$, then $[a,b]\in \Sing(\calD_2)$ and $\Sigma$ contains the point 
$([a,b],[c,d])$, where $c^2+t^kd^2 = 0$. It is a singular point of $\Sigma$. If $n = 0$ and $k > 1$, then 
$[a,b]\in \Sing(\calD_2)$ and 
the point $t = x= 0, y= 1$ is a singular point.

Conversely, if $[a,b]\in \Sing(\calD_2)$, then $2n+k \ge 2$, and we can reverse the argument to show that 
$\Sigma$ 
is singular over $t = 0$ if and only if $2n+k\ge 2$.
\end{proof}


All singular points of  $\Phi_d$ outside $\Gamma$ define singular points of the 
conic bundle $\Phi_d^{\bp}$. Other singular points of $\Phi_d^{\bp}$ may lie on the curve $\Sigma$.
 
Taking partials of \eqref{eq:conicbundle} and setting $y_0 = 0$, we obtain the following.

\begin{lem}\label{lem:singDelta} The following assertions are equivalent:
\begin{itemize}
\item[(i)] The rank of the matrix
\beq\label{matrixdelta}
\begin{pmatrix} B_{d-2}(a,b)& D_{d-2}(a,b)&E_{d-2}(a,b)\\
C_{d-2}(a,b)&E_{d-2}(a,b)&F_{d-2}(a,b)
\end{pmatrix}
\eeq
at the point $[a,b]$ is equal to one.
\item[(ii)] The residual conic $C_{a,b}$ over a point 
$(a,b)\in \calD_2$ has a singular point  
$[0,0,-E_{d-2}(a,b),D_{d-2}(a,b)] = [0,0,-F_{d-2}(a,b),E_{d-2}(a,b)]\in \Gamma$.
\item[(iii)] A point on $\Sigma$ over $[a,b]$ is a singular point of a fiber of the conic bundle.
\end{itemize}
If $d\ge 4$, then these conditions are also equivalent to the condinition
\begin{itemize}
\item[(iv)] The surface $\Phi_d^{\bp}$ has a singular point on $\Sigma$ over $[a,b]\in \calD_2$.
\end{itemize}
\end{lem}

\begin{proof} We only prove that (i) implies (ii) and leave to prove other implications to the reader. 
Suppose (i) holds. then $D_{d-2}(a,b)F_{d-2}(a,b)-E_{d-2}(a,b)^2 = 0$, hence $[a,b]\in \calD_2$. Let $(0,y_1,y_2)$ be the 
coordinates of the intersection point of the fiber of conic bundle $K_{a,b}$ over $[a,b]$ with $\Sigma$. 
Then $y_0B_{d-2}(a,b) = y_0C_{d-2}(a,b) = 0$. If $B_{d-2}(a,b)\ne 0$ or $C_{d-2}(a,b)\ne 0$, then $y_0 = 0$, and hence $[0,y_1,y_2]$ 
is a singular point of $K_{a,b}$. If $B_{d-2}(a,b)= C_{d-2}(a,b) = 0$, 
then $P(a,b) = A_d(a,b)(D_{d-2}(a,b)F_{d-2}(a,b)-E_{d-2}(a,b)^2) = 0$ 
vanishes again. Thus the conic has equation $y_0^2A_d(a,b)+(E_{d-2}(a,b)y_1+F_{d-2}(a,b)y_2)^2 = 0$, hence again the point 
$(0,y_1,y_2)$ is its singular point.

If $d\ge 3$, taking the partials in $y_0,y_1,y_2,t_0,t_1$ and using that $D_{d-2},E_{d-2}$ and $F_{d-2}$ are  binary 
forms of degree $\ge 2$, we get condition (iv).
\end{proof} 

By analogy with the case of cubic surfaces, a singular point of a residual conic lying on 
$\Gamma$ may be called an \emph{Eckardt point} of $\Phi_d$.

 \begin{prop} $\Phi_d^{\bp}$ is a normal surface if and only if $P(t_0,t_1) \ne 0$ and $\Phi_d$ has no singular lines intersecting 
 $\Gamma$ (besides 
 $\Gamma$).
 \end{prop}

\begin{proof} Since $\Phi_d^{\bp}$ is a hypersurface in a smooth variety, $\Phi_d^{\bp}$ is normal if and only if 
it has only isolated singular points. 

Suppose $\calD_1\ne \Gamma^\perp$, i.e. $P(t_0,t_1) \ne 0$. Thus a general residual conic is smooth. 
Taking a general plane section of $\Phi_d$, we obtain 
that $\Sing(\Phi_d\setminus \Gamma)$ consists
of isolated singular 
points or lines intersecting $\Gamma$. By the assumption, $\Phi_d$ does not have such lines. 
So, if $\Phi_d^{\bp}$ is non-normal, the curve $\Sigma$ must be singular along its irreducible component.
By Lemma \ref{lem:singDelta}, the rank of matrix \eqref{matrixdelta} is identically equal to $1$. This implies that 
the rank of matrix \eqref{discriminant} is equal to 2, hence $P(t_0,t_1) = 0$, a contradiction.

Conversely, suppose $\Phi_d^{\bp}$ is normal. Then a general residual conic is irreducible, hence its proper transform 
in $\Phi_d^{\bp}$ is irreducible, hence smooth. Thus $P(t_0,t_1)\ne 0$. The second assumption is obviously satisfied.
\end{proof}



\begin{prop}\label{sarkisov} 
$\Phi_d^{\bp}$ is smooth if and only if the discriminant polynomial $P(t_0,t_1)$ has no multiple roots. 
\end{prop}

\begin{proof} This can be deduced from the well-know properties of conic bundles over any regular base 
(see \cite{Sarkisov}). In our situation, when the base is one-dimensional, the proof is very easy and 
we will present it here.
After localizing at a point on the base, we may assume that $X\subset \bbP_R^2$, where $R$ is a discrete valuation 
ring with local parameter $t$. We use again that a quadratic form over $R$ can be 
reduced to a sum of squares. Thus $\Phi_d$ can be given by an equation 
\beq\label{kneser}
u^2+\epsilon_1t^nv^2+\epsilon_1 t^{m}w^2 = 0,
\eeq
where $\epsilon_i$ are units and $n\le m$. Taking the partials, we see that $(t;u,v,w) = [0,0,0,1]$ is 
a non-regular point of $\Phi_d$ if 
$m > 1$. Thus, if  $\Phi_d$ is regular, the discriminant $D = \epsilon_1\epsilon_2t^{n+m}$ is of order $\le 2$.
Conversely, if $D$ is of order $1$, then $(n,m) = (0,1)$, then $\Phi_d$ is regular and the fiber of $t = 0$ is 
a line-pair. If the order of $D$ is equal to 2, then $(n,m) = (1,1)$ or $(0,2)$. In the first case $\Phi_d$ is non-regular 
at $[0,0,v_0,w_0]$, where $\epsilon_1(0)u_0^2+ \epsilon_2(0)v_0^2 = 0$ and the fiber 
over $t = 0$ is a double line $u^2 = 0$. In the second case $x = [0,0,0,1]$ is a non-regular point of $\Phi_d$, 
 the fiber over $t = 0$ 
is a line-pair with singular point at $x$. 
\end{proof}

\begin{remark} In the classical terminology, the image in $\Phi_d$ of a non-reduced  fiber of the conic bundle 
$\Phi_d^{\norm}$ is 
 a \emph{torsal line}. In general, a torsal line is a line on a surface with the same tangent plane at each 
 smooth point of the surface lying on the line. 
\end{remark}  

Using local equations \eqref{kneser}, we can analyze possible singularities of $\Phi_d$.

Suppose $n = 0$, then the singular point is formally isomorphic to the singularity of $xy+z^{m} = 0$ at the origin.
It is a rational double point of type $A_{m-1}$.  Note that in this case the point is a singular point of a reduced conic fiber.

Suppose $n > 0$. Then the fiber over $t= 0$ is a double line. If $n = m = 1$ (and hence $[a,b]$ is a double root of $P(t_0,t_1)$), we have two ordinary double points 
on  it. If $n > 1$, $[a,b]$ is a root of multiplicity $> 2$ of $P(t_0,t_1)$ and the surface is non-normal. 
If $1\le n< m$, the singularity is formally isomorphic to the singularity of
 $x^2+t^ny^2+t^{n+k} = 0, k > 0,$ at the origin. If $k$ is odd, the singularity is quasi-homogeneous with weights 
 $(q_1,q_2,q_3) = (n+k,2,k)$ and degree $d = 2n+2k$. It is known that 
 it is a rational double point if and only if $d-q_1-q_2-q_3 = n-2< 0$. When $n = 0$, it is an ordinary node if and only if 
 $k = 2$ (and hence 
 $[a,b]$ is a root of multiplicity $2$ of $P(t_0,t_1)$). Otherwise it is a rational double point of type $A_{k-1}$.
 If $n = 1$, it is a rational double point of type $D_{n+k+1}$.
 
\begin{definition} We say that $\Phi_d$ is non-degenerate if the following conditions are satisfied:
\begin{itemize}
\item[(i)] $\Phi_d^{\bp} = \Phi_d^{\norm}$;
\item[(ii)]$P(t_0,t_1)$ has no roots of multiplicity larger than $2$;
\item[(iii)] $R(t_0,t_1)$ has no multiple roots.
\end{itemize} 
\end{definition}

The following proposition follows from the previous analysis of possible singularities of $\Phi_d$.

\begin{prop}\label{prop:nondegenerate} Suppose $\Phi_d$ is non-degenerate. Then  
\begin{itemize}
\item[(i)] the curve $\Sigma$ is a smooth curve of genus $d-3$;
\item[(ii)] no singular point of $\Phi^{\norm}$ lies on $\Sigma$;
\item[(iii)] all isolated singular points of $\Phi_d$ are ordinary nodes;
\item[(iv)] every singular point of $\Phi_d^{\norm}$ is a singular point of some fiber of the conic bundle; 
\item[(v)] each double fiber of the conic bundle on $\Phi_d^{\norm}$ has two singular points.
\end{itemize}
Conversely, properties (i), (ii) and (iii) imply that $\Phi_d$ is non-degenerate.
\end{prop}

\begin{proof} Suppose that $\Phi_d$ is non-degenerate. Property (i) follows from the definition and 
Lemma \ref{lem:singDelta}. If $x$ is singular point of $\Phi_d^{\norm}$ lying on $\Sigma$, then the exceptional divisor 
$E$ of the blow-up of $\Gamma$ intersects $\Sigma$ at a singular point of $\Phi_d$, hence this point must be a singular point of 
$\Sigma$ contradicting (i).

(iii) Since $P(t_0,t_1)$ has no roots of multiplicity larger than $2$, the analysis 
of possible singular points of $\Phi_d$ in the proof of Proposition \ref{sarkisov} shows that all singular points of 
$\Phi^{\norm}$ are ordinary nodes. This shows that (iii) follows from (ii).

(iv) Follows from the analysis of singularities of $\Phi_d$ in the proof of Proposition \ref{sarkisov}.

Let us prove the converse. Property (i) implies that $\Sigma$ is reduced and hence 
$\Phi_d^{\bp} = \Phi_d^{\norm}$. It also implies 
that $R(t_0,t_1)$ has no multiple roots.  Properties (ii) and (ii), together 
with the analysis of singularities of $\Phi_d$ from Proposition \ref{sarkisov}, imply that the polynomial $P(t_0,t_1)$ has no roots 
of multiplicity higher than $2$.
\end{proof}

 \begin{remark} One can show that, only assuming that $\Phi_d^{\bp} = \Phi_d^{\norm}$, all singularities of $\Phi_d$ are rational double  
 points.
 \end{remark}

\begin{corollary} Suppose $d > 3$. The number of isolated singular points on a non-degenerate  surface $\Phi_d$ 
is at most $3d-4$.
\end{corollary}

\begin{proof} Every singular point lies over a singular point of the discriminant $\calD_1$, 
and two singular points may lie over the same point if the corresponding fiber is not reduced. This gives 
the asserted bound. 

\end{proof}

\subsection{$4$-nodal cubic surfaces}
A normal cubic surface $\Phi_3$ with a fixed line $\Gamma$ on it is a monoidal surface. 

 Each reducible residual conic is cut out by a tritangent plane 
containing the line $\Gamma$. The surface is smooth if and only if there are 5 different 
tritangent planes corresponding to the roots of the binary form $P(t_0,t_1)$ of degree 
$3d-4 = 5$. The roots 
of the binary form $R(t_0,t_1)$ of degree 2 correspond to the residual conics intersecting $\Gamma$ at one point.
The common roots of $P(t_0,t_1)$ and $R(t_0,t_1)$ correspond to Eckardt points on $\Gamma$.

The surface $\Phi_3$ is non-degenerate if and only if all its singular points are ordinary nodes and they do not lie on 
$\Gamma$. It is isomorphic to the blow up of the set $\Sigma$ of six points in the plane.
Since $s_1+2s_2+2s_3 = 5$, we have the following possibilities 
$(s_1,s_2,s_3) = (5,0,0), (3,1,0), (3,0,1), (1,2,0), (1,0,2), (1,1,1)$. 
The number of nodes is equal to $0,1,2,2,4,3$, respectively. It is known that a normal cubic surface contains at most 
four nodes, and the surface with four nodes is unique, up to isomorphism. A monoidal cubic surface with $(s_1,s_2,s_3) 
= (1,0,2)$ gives a birational model of this surface
as the blow-up of six points $p_0,p_1,p_2,p_3,p_4,p_5$ with two infinitely near points 
$p_3\succ p_2$ and $p_5\succ p_4$ lying on the lines $\la p_1,p_2\ra$ and $\la p_1,p_4\ra$. The Cremona transformation 
defined by the homaloidal linear
system of cubic curves $C^3(p_0^2,p_2,p_3,p_4,p_5)$ passing simply through $p_1,\ldots,p_5$ and passing through 
$p_0$ with multiplicity two defines
 a biregular model of $\Phi_d$ as the blow-up of the vertices of a complete quadrilateral. This is a familiar 
 model of a minimal resolution of the 4-nodal cubic surface.

\subsection{Pl\"ucker Complex Surface}
Non-degenerate quartic surfaces with a double line are discussed in Jessop's book 
\cite[Chapter VI]{Jessop} and in Salmon's book \cite[Chapter XVI]{Salmon}. 
Let us see how some of the classical results follow from our discussion. 
Note that Jessop and Salmon only implicitly assumed that the surfaces $\Phi_4$ are non-degenerate. 

There are $s_2+2s_3 = 8-s_1-s_2$ ordinary nodes on $\Phi_4$. The maximal number is equal to 8. It is 
realized on the \emph{Pl\"ucker's Complex Surface},\footnote{Here Complex 
refers to a quadratic complex of lines in $\bbP^3$, see Remark \ref{plucker}.} a submonoidal quartic 
surface with $8$ nodes (see \cite[Section 83]{Jessop}). The nodes lie on four torsal lines. So, the discriminant 
polynomial 
$P(t_0,t_1)$ is a square. The four planes $\Pi_t$ corresponding to the roots of $P(t_0,t_1)$ 
intersect $\Phi_4$ along a torsal line and the double line. 

Let $P_i,P_i', i = 1,\ldots,4,$ be the four pairs of nodes lying on torsor lines
Projecting from one of them, say $P_1$, we find that the branch sextic splits into the union of a double line $\ell$ 
(the projection of $\Gamma$) and a complete quadrilateral $T$ whose vertices are the projections of the points.
The image of the exceptional curve over $P_1$ is a \emph{contact-conic}, a conic tangent to all sides 
of the complete quadrilateral and to the double line. 
The tangency point on the double line  
is the image of the torsal line containing $P_1$. The sides of  $T$ are the images of eight  
tropes intersecting each torsal line at one of the nodes. The configuration of $8$ tropes and $8$ nodes is a symmetric configuration 
$(8_4)$ (i.e. each node is contained in four tropes and each trope contains four nodes).  
For example, we can find this configuration by taking eight vertices 
of a cube with 6 faces and two diagonal planes. 



By a well-known Cayley's characterization of quartic symmetroids as double planes branched along the union of 
two plane cubic curves with a contact-conic, we obtain that a Pl\"ucker Complex Surface 
admits an equation given by a determinant of a symmetric matrix with linear forms as its entries. One of 
these equations can be found in \cite{Helso}:
\beq
\det\begin{pmatrix}0&x_0-x_1+x_2&x_0-x_1+x_2&x_0\\
x_0-x_1+x_2&0&x_3&x_1\\
x_0-x_1+x_2&x_3&0&x_2\\
x_0&x_1&x_2&0\end{pmatrix} = 0,
\eeq
The double line is $V(x_0-x_1,x_2)$. The four torsal lines are 
cut out by planes $V(x_0-x_1+tx_2)$, where 
$t = \frac{1+\sqrt{-3}}{2},  \frac{1-\sqrt{-3}}{2},1,\infty$. They are the intersections of these planes with the planes 
$V(x_0+x_3), V(x_0+x_3),V(x_0+x_2-x_3),V(x_2-x_3)$, respectively. The eight nodes lie by pairs in the torsal lines
\begin{eqnarray*}
(P_1,P_1') &= &([0,e,1,0],[-1,0,\bar{e},1]),\quad
(P_2.P_2')=([0,\bar{e},1,0],[-1,0,e,1]),\\
(P_3.P_3')&=& ([0,1,1,1],[1,0,-1,0]), \quad 
(P_4,P_4')=([1,0,0,0],[0,1,0,1]).
\end{eqnarray*}
where $e = \tfrac{1+\sqrt{-3}}{2}, \bar{e} = \tfrac{1-\sqrt{-3}}{2}$. 

The eight tropes are 
\begin{eqnarray*}
V(x_0)&=&\la P_1,P_2,P_3,P_4'\ra,\\
V(x_1)&=&\la P_1',P_2',P_3',P_4\ra,\\
V(x_3)&=&\la P_1,P_2,P_3',P_4\ra,\\
V(x_0+x_3-x_1)&=&\la P_1',P_2',P_3,P_4'\ra,\\
V(ex_2+\bar{e}x_3-x_1)&=&\la P_1,P_2',P_3,P_4\ra,\\
V(\bar{e}x_2+ex_3-x_1)&=&\la P_1',P_2,P_3,P_4\ra,\\
V((x_0+x_2)-e(x_2-x_3))&=&\la P_1,P_2',P_3',P_4'\ra,\\
V((x_0+x_2)-\bar{e}(x_2-x_3))&=&\la P_1',P_2,P_3',P_4\ra.
\end{eqnarray*}

\begin{remark}\label{plucker} The Pl\"ucker's Complex Surface occurs in the geometry of a quadratic 
line complex $\frakC$ in the Grassmann variety 
$G_1(\bbP^3)$ of lines in $\bbP^3$ (see \cite[Art 86]{Jessop}, \cite[Art 455]{Salmon}). 
Fix a general line $\Gamma$ in $\bbP^3$. The set of rays in $\frakC$ contained in a general plane 
$\Pi$ is a conic  $K_\Pi$ in $G_1(\bbP^3)$. A pencil of lines $\Omega(x,\Pi)$ contained in $\Pi$ and passing through a point 
$x$ is line in $G_1(\bbP^3)$. 
A Pl\"ucker Complex Surface $\Phi_4$ of $\frakC$ 
is defined to be the locus of points $x\in \bbP^3$ such that the line $\Omega(x,\la x,\Gamma\ra)$ is tangent to the
conic $K_{\la x,\Gamma\ra}$.
 There will be four planes $\Pi_1,\ldots,\Pi_4$ in $\Gamma^\perp$ such that the conic $K_\Pi$ is reducible. 
 They correspond 
 to the intersection points of the pencil $\Gamma^\perp$ with the dual of the Kummer 
 surface associated with $\frakC$.
 The lines corresponding to their singular points are the torsal lines, and the irreducible 
 components of $K(\Pi_i)$ define two singular points on the trope.
 
\end{remark}

\section{A blow-up model of a submonoidal surface}

\subsection{Blowing down to a minimal ruled surface} 
From now on, we will assume that $\Phi_d$ is a nondegenerate submonoidal surface.


Let $\Phi_d^{\norm} = \Phi_d^{\bp}\subset \Bl_\Gamma(\bbP^3)$ be the normalization of a submonoidal surface 
$\Phi_d$. Let 
$f:\tilde{\Phi}\to \Phi_d^{\norm}\to \Phi_d$ be a minimal resolution of singularities. 
The composition $\tilde{\pi} = \pi\circ f$ defines a structure of a conic bundle on $\tilde{\Phi}$. There 
are three kinds of singular fibers of the conic bundle.

A fiber of the \emph{first kind} is the union of two different components which 
intersect at a nonsingular point
of $\tilde{\Phi}_d$.  A fiber 
of the \emph{second kind} is the pre-image of a reduced residual conic with singular point at a node of $\Phi$. 
It is a divisor $\bar{F}_1+\bar{F}_2+R$, where $\bar{F}_i$ are the proper transforms of the irreducible components of 
$F$ and $R$ is the exceptional curve. The curves $\bar{F}_i$ are $(-1)$-curves and the curve $R$ is a $(-2)$-curve. 
Here we use the standard abbreviation for a smooth rational curve on a smooth projective surface with 
negative self-intersection $-n$. A fiber of the \emph{third kind} is the pre-image of a non-reduced fiber of $\pi$. 
Each fiber of the third kind is a divisor $R_1+R_2+2\bar{F}$, where 
$R_1$ and $R_2$ are the exceptional fibers over the singular points lying on the fiber and 
$2\bar{F}$ is the proper transform of the fiber.

Let $s_1$ (resp.~$s_2$, resp.~$s_3$) be the number of fibers of the first (resp. second, resp. third) kind.

 We have
$$3d-4 = s_1+2s_2+2s_3.$$

We have $s_1$ fibers $F$ with the Euler-Poincar\'e characteristic $e(F) = 3$ and $s_2+s_3$ fibers with $e(F) = 4$. 
The formula for the Euler-Poincar\'e characteristic of a fibered surface gives 
$$e(\tilde{\Phi}) = 4+s_1+2(s_2+s_3) = 3d.$$
By Noether Formula,  
$$K_{\tilde{\Phi}}^2 = 12-3d.$$
Since all singular points of $\Phi_d$ are rational double points, the adjunction formula for a surface in $\bbP^3$ 
\cite[Appendix to Chapter III]{Zariski} gives 
\beq\label{adjunction}
K_{\tilde{\Phi}} = (d-4)H-(d-3)\Sigma.
\eeq

A straightforward computation shows that 
$$\Sigma{}^2 = d-4.$$
Let  $\phi:\tilde{\Phi}\to \bfF_n$ be a birational map to a minimal ruled surface $\bfF_n$. It blows down  
one irreducible component $R_1^{(1)},\ldots,R_{s_1}^{(1)}$ in each fiber $R_i^{(1)}+R_i'{}^{(1)}$ of the first kind, 
two $(-1)$-curves $R_i^{(2)}+R_i'{}^{(2)}, i = 1,\ldots,s_2,$ in each fiber $R_i^{(2)}+R_i'{}^{(2)}+Q_i^{(2)}$ of the second kind, and 
two components $R_i^{(3)}+R_i'{}^{(3)}, i = 1,\ldots,s_3,$ in fibers $R_i^{(3)}+Q_i{}^{(3)}+Q_i'{}^{(3)}$ of the 
third kind, where
$R_i^{(3)}$ is a $(-1)$-curve.

We check that $e(Z) = 3d-s_1-2s_2-2s-3 = 4$, as it should be.
Let 
$$\calR = \sum_{i=1}^{s_1}R_i^{(1)}+\sum_{i=1}^{s_2}(R_i^{(2)}+R_i'{}^{(2)})+2\sum_{i=1}^{s_3}R_i^{(3)},$$
and
$$A=d\Sigma-(d-1)H+\calR.$$
It follows from Proposition \ref{prop:nondegenerate} that $\Sigma$ does not pass through singular points of 
$\Phi_d^{\norm}$, hence $\Sigma'$ and $H$ intersect each irreducible component of $R$ 
with the same multiplicity equal to $1$.
 This immediately implies that 
$A\cdot \calR = 0$. 
We also get 
\begin{eqnarray*}
A\cdot H &=& -(d-1)d-d(d-2)+s_1+2s_2+2s_3 = 2d-4,\\
A\cdot \Sigma &=& -(d-1)(d-2)+d(d-4)+(3d-4) = 2d-6.
\end{eqnarray*}
It follows that 
$$A^2 = (d-1)A\cdot H-(d-1)A\cdot \Sigma = (d-1)(4-2d)-d(-2d+6) = -4.$$
Let $\sigma(A)$ be the image of $A$ on $\bfF_n$ and $[\sigma(A)] = a\frakf+b\frake$. 
Since $A\cdot \frakf = 2d-2(d-1) = 2$, we get $b = 2$. Since $A^2 = \sigma(A)^2 = -4$, we get 
$4a+4n^2 = -4$, hence $a=0, n= -1$. This shows that 
$$n = 1, \quad \sigma(A) = 2\frake,$$
where we identify $\frake$ with the exceptional section.

Let $\tilde{\Phi} \to \bfF_1\to \bbP^2$ be the composition of $\phi$ with the blowing down of the exceptional section $\frake$.
The surface $\Phi_d^{\norm}$ becomes isomorphic to the blow-up of a set of  
$1+(s_1+2s_2+2s_3) = 3d-3$ points 
$$\calP = (p_0,p_1,\ldots,p_{s_1}, q_{1},q_1',\ldots,q_{s_2},q_{s_2}', r_1, r_1',\ldots,r_{s_3},r_{s_3}')$$
in the plane, where 
\begin{itemize}
\item $p_0$ is the image of the exceptional section $\frake$, 
\item $r_i'\succ r_i$ is infinitely near to $r_i$,
\item $p_0,r_i,r_i'$ are collinear,
\item $p_0,q_i,q_i'$ are collinear.
\end{itemize}

Let 
$$(e_0,e_1,\ldots,e_{3d-4}, e_{3d-3})$$
be the standard geometric basis of 
$\Pic(\tilde{\Phi})$ corresponding to the blow-up of the ordered set of points $\calP$. Here we denote by 
$e_{3d-3}$ the divisor class of the exceptional curve over $p_0$.
  We find that the conic bundle on $\tilde{\Phi}$ is given by the linear 
system $|e_0-e_{3d-3}|$, its members are the pre-images of lines passing through the point $p_0$.
 
 We also find that
\beq\label{H}
[H] = de_0-\sum_{i=1}^{3d-4}e_i-(d-2)e_{3d-3},
\eeq
\beq\label{Delta'}
[\Sigma] = (d-1)e_0-\sum_{i=1}^{3d-4}e_i-(d-3)e_{3d-3}.
\eeq

\begin{lem} 
$$|(d-1)e_0-\sum_{i=1}^{3d-4}e_i-(d-3)e_{3d-3}| = \{\Sigma\}.$$
\end{lem} 

\begin{proof} Suppose the linear system contains a pencil. Since 
$$de_0-\sum_{i=1}^{3d-4}e_i-(d-2)e_{3d-3} = ((d-1)e_0-\sum_{i=1}^{3d-4}e_i-(d-3)e_{3d-3})+(e_0-e_{3d-3}),$$
it maps $\tilde{\Phi}$ to 
$\bbP^1\times \bbP^1$ embedded into $\bbP^3$ as a nonsingular quadric. This contradiction proves the lemma.
\end{proof}

The following theorem sums up what we have found.

\begin{thm}\label{thm:5.3} Let $\tilde{\Phi}_d$ be a minimal resolution of the nodes on the normalization 
$\Phi_d^{\norm}$ of a non-degenerate submonoidal surface $\Phi_d$ of degree $d$. 
Then $\tilde{\Phi}_d$ is isomorphic to the blow-up $\Bl_\calP(\bbP^2)$ of a set $\calP$ of $3d-3$ points in $\bbP^2$.
The linear system $|de_0-\sum_{i=1}^{3d-4}e_i-(d-2)e_{3d-3}|$ defines a birational morphism
$\phi: \tilde{\Phi}_d\cong \Bl_\calP(\bbP^2)\to \Phi_d.
$
The singular line $\Gamma$ of $\Phi_d$ is the image 
of the unique curve $\Sigma \in |(d-1)e_0-\sum_{i=1}^{3d-4}e_i-(d-3)e_{3d-3}|$.
The residual conics on $\Phi_d$ are the images of the members of the pencil $|e_0-e_{3d-3}|$. 
The proper transform of a line 
$\la p_0,q_i,q_i'\ra$ is mapped to a node $P_i$ on $\Phi_d$ equal to the singular point 
of the residual conic in 
$\Phi_d\cap \la \Gamma,P_i\ra$. Other nodes lie in pairs on the  torsal lines on $\Phi_d$ equal 
to the images of the exceptional curves over $r_i'$.  
\end{thm}

Note that the plane blow-up models of submonoidal surfaces depend on a choice of $3d-3$ points in the 
plane modulo projective 
equivalence. It depends on $2(3d-3)-8 = 6d-14$ parameters. This agrees with the number of parameters 
of submonoidal surfaces.

\subsection{Special sections of the conic bundle}
The blow-up model allows us to give simple proofs of many facts about submonoidal surfaces which can be found 
in \cite[\S 131]{Sturm}. For example, we get the following. 

\begin{thm}\label{cor:5.4} A submonoidal surface $\Phi_d$ contains a smooth rational curve 
$S$ of degree $d-2$ that intersects $\Gamma$ at $d-3$ points and defines a section $E$ of the conic bundle 
$\tilde{\Phi}_d\to \Gamma^\perp$. If 
$\Phi_d^{\norm}$ is smooth, there are $2^{3d-5}$ such curves.
\end{thm}

\begin{proof} Suppose such curve exists. Let us see what would be its pullback $E$ in 
$\tilde{\Phi}_d = \Bl_{\calP}(\bbP^2)$. Using 
\eqref{adjunction}, we get 
$E\cdot K_{\tilde{\Phi}} = E\cdot ((d-4)H-(d-3)\Sigma) = (d-4)(d-2)-(d-3)^2 = -1$. 
Since $E$ is rational, we get $C^2 = -1$, so $C$ is a 
$(-1)$-curve on $\tilde{\Phi}$. We must also have $E\cdot H = d-2$ and $E\cdot \Sigma = d-3$. This implies 
that $E\cdot (e_0-e_{3d-3}) = C\cdot (H-\Sigma) = 1$. This shows 
that the image of $E$ is a section of the conic bundle. Also it suggests  
 a construction of all possible curves $S$ from the assertion of the Corollary. 
They must be the images 
of $(-1)$-curves $E$ on $\tilde{\Phi}$ satisfying $E\cdot H = d-2$ and $E\cdot \Sigma = d-3$. If 
we write the divisor class of $E$ as $ne_0-\sum_{i=1}^{3d-4}n_ie_i-ke_{3d-3}$, we must have
\begin{eqnarray}\label{special}
&&n^2-\sum_{i=1}^{3d-4}n_i^2-k^2 = -1, \\ \notag
&&3n-\sum_{i=1}^{3d-4}n_i-k = 1,\\ \notag
&&dn-\sum_{i=1}^{3d-4}n_i-(d-2)k = d-2,\\ \notag
&&(d-1)n-\sum_{i=1}^{3d-4}n_i-(d-3)k = d-3.
\end{eqnarray}
Subtracting the fourth equation from the third one, we obtain $n = k+1$. Substituting this in 
the first and the second 
equations, we get
$$2k = \sum_{i=1}^{3d-4}n_i^2 = \sum_{i=1}^{3d-4}n_i.$$
This implies that $n_i\in \{0,1\}$. Let $I = \{i\in [1,3d-4], n_i=1\}$. We get $2k = \#I$, and third equation 
gives $n = k+1$. from this deduce that 
\beq\label{specialsection}
E \in |ne_0-\sum_{i\in I}e_i-(n-1)e_{3d-3}|,
\eeq
where $\#I = 2n$.  
The divisor class $e$ on a rational surface $X$ with $e^2=E\cdot K_X = -1$ is not necessary represented by a $(-1)$-curve.
However, since $E\cdot (e_0-e_{3d-3} = 1$, if $E$ is reducible, then one of its irreducible components $E'$ is a section, and other 
irreducible components are blown down to singular points of $\Phi_d$.  So, if $\Phi_d^{\bp}$ is smooth, all 
the divisor closes $E$ define special sections, and their number is equal to the number of subsets of $[1,\ldots,3d-4]$ of even cardinality, it is the claimed number. 
\end{proof}

We call $S$ a \emph{special section} of the conic bundle. Its degree (with respect to the blowing down 
$\tilde{\Phi}_d\to \bbP^2$) is equal to $E\cdot e_0 = n$.

 We know that $\Phi_d^{\bp}$ is not smooth only if either some paints $p_0,p_a,p_b$ in $\calP$ are collinear, or 
 there is an infinitely near point $p_{i+1}\succ p_i$. The divisor class 
 $e_0-e_a-e_b-e_{3d-3},$ (resp. $e_0-e_{a}-e_{a+1}-e_{3d-3}$) is the class of a  $(-2)$-curve $R_{a,b}$ (resp. $R_a$) that is blown down to a singular point.
 Let $E = E'+R_1+\cdots+R_k$, where $E'$ is a section and $R_i$ are disjoint $(-2)$-curves blown down to singular points 
 (recall that, by assumption, 
 $\Phi_d$ is non-degenerate and hence all singular points are ordinary nodes). 
 Since
 $$E\cdot R_i = (E'+R_1+\cdots+R_k) \cdot R_i = -1,$$
 we find that $R_i = R_{a,b}$ (resp. $R = R_a$) if and 
 only if $a,b\in I$ (resp. $a,a+1\in I$). The divisor class of  $E'$ is still given by formula \eqref{specialsection} 
 only with maybe smaller $n$.

 Since the torsal line corresponds to a double fiber of the conic bundle, a special section $S$ passes through one of 
 the singular points 
of $\Phi_d$ lying on it. It may or may not pass through other singular points of $\Phi_d$.

 Assume $d = 2m$. Let $S$ be a special section of $\Phi_d$ equal to the image of an exceptional curve 
$E$ from \eqref{specialsection}. Let  
\beq\label{edag}
[E^\dag] \in  |(m-1)H-(m-2)\Sigma-E| = |(3m-2-n)e_0-\sum_{i\in I}e_i-(3m-n-3)e_{3d-3}|.
\eeq

We check that $E^\dag$ satisfies \eqref{special}, and hence its image in $\Phi_d$ is a special section $S^\dag$.

We call it the \emph{dual special section}. It follows from the definition that $S^{\dag}$ is the residual curve 
cut out by a monoidal surface $\Phi_{m-1}'$ that contains $S$. Counting parameters, we obtain that there is a unique such surface.
For example, if $m = 2$, the union of a special section and its dual section is a plane section of 
$\Phi_4$. 

\begin{example}\label{quartics} Suppose $d = 4$. The surface $\tilde{\Phi}_4$ is isomorphic to the blow-up of 
a set $\calP$ of $9$ points $p_0,\ldots,p_8$ in $\bbP^2$. The linear system $|H|$ is given by curves 
$C^4(p_0^2,p_1,\ldots,p_8)$ and the curve 
$\Sigma'$ is a cubic $C^3(p_0,p_1,\ldots,p_8)$. If $\tilde{\Phi}_4 = \Phi_4^{\norm}$, then we have $2^7$ 
special sections. They are conics in $\Phi_4$ intersecting $\Gamma$ at one point. 
The union of a special sections $S$ and its dual special section $S^\dag$ is a plane section of $\Phi_d$.
Special sections of degree $1$ are the proper transforms of 
lines $\la p_i,p_j\ra, 0 < i < j\le 8$. There are $28$ of them. There are $70$ 
special sections of degree $2$. They are the proper transforms of conics in the plane passing through 
$p_0$ and a subset of four points in $\calP\setminus \{p_0\}$. Special sections of degree $3$ are dual to special sections of degree 
one, they are the proper transforms 
of cubics with a node at $p_0$ and passing through 6 points in $\calP\setminus \{p_0\}$. There is one special section 
of degree $4$, the proper transform of the quartic through $\calP$ with a triple point at $p_0$. Its dual special section 
is of degree $0$, the exceptional curve over $p_0$.
If $\Phi_4^{\norm}$ is not smooth, some of these conics become reducible or coincide. This happens when the 
points $p_0,\ldots,p_8$ are not in a general position. For example, the Pl\"ucker Complex surface contains only 
nine special sections.
\end{example}

\subsection{Special sections and congruences of lines}
Let $S \subset \Phi_d$ be a special section. The closure of the set of lines in $\bbP^3$ intersecting $S$ and $\Gamma$ at one point is a congruence 
of lines $\calC_S$ in the Grassmannian $G_1(\bbP^3)$ of lines in $\bbP^3$. Its order  (= the number of rays of the congruence 
passing through a general point in $\bbP^3$) is equal to one and its class (= the number of rays in a general plane) 
is equal to $d-2$. It is one of the three possible kinds of congruences of order one in $\bbP^3$ in 
Kummer's classification \cite{Kummer} (see a modern exposition of this classification in \cite{Arrondo} or \cite{Ran}). 
In the Pl\"ucker embedding, $\calC_S$ is a ruled surface of degree $d-1$ contained in the special line complex 
$\Omega(\Gamma)$. Its generators are the pencils of lines $\Omega(x,\Pi)$ in a plane $\Pi\in \Gamma^\perp$ that 
pass through the point $x = S\cap \Pi$. 

We know that the monoidal surfaces $\Phi_n$ which contain $\Gamma$ with multiplicity $n-1$ depend on $3n$ parameters.
They intersect $S$ with multiplicity $n-1$ at $d-3$ points lying on $\Gamma$. If we require that, additionally, they pass 
through $n(d-2)-(d-3)(n-1)+1 = n+d-2$ general points in $S$, we obtain that $\Phi_n$ contains $S$. Generically, 
the minimal $n$ for which this is possible 
is equal to $\lceil \frac{d-2}{2}\rceil = m-1$, where $d = 2m+1$ or $d= 2m$. We expect that 
there will be a unique such $\Phi_{m-1}$ if $d$ is even, and 
a pencil if $d$ is odd. 

There is a unique line on $\Phi_{m-1}$ which is contained in a generator $\Omega(x,\Pi)$ of $\calC_S$. 
This shows that generatrix of the ruled surface $\Phi_m$ in $\bbP^3$ is a directrix of degree $m$ 
of the ruled surface $\calC_S$.  Thus $\calC_S$ is a ruled surface of degree $d-1$ that has 
a directrix of degree $m-1$. 

It follows from Lemma \ref{directrix} that $m-1$ is expected to be the minimal degree of a directrix of the ruled surface 
$\calC_S$. For a special position of points $p_1,\ldots,p_{3d-3}$, the smallest degree of a directrix of $\calC_S$, 
and hence of a section of the conic bundle, could be less 
than $\lceil \frac{d-1}{2}\rceil$. For example, if $A_d = 0$ in equation \eqref{eq:Phi}, 
there is a special section of degree 1 given by $x_2=x_3 = 0$.

A general point in $\bbP^3$ is contained in a unique ray of $\calC_S$ that 
intersects $\Phi_d$ at one point outside $\Gamma\cup S$. This defines a birational map 
$h:\Phi_d\da \calC_S$ over $\Gamma$ that can be regularized by 
the following commutative 
diagram
\beq
\scalebox{0.7}{\xymatrix{\tilde{\Phi}_d\ar[dr]\ar[rrrr]^{\tilde{h}}\ar@/_/[ddrr]&&&&\calC_S^{\norm}\ar[dl]\ar@/^/[ddll]\\
&\Phi_d\ar@{-->}[dr]\ar@{-->}[rr]^h&&\calC_S\ar@{-->}[dl]&\\
&&\Gamma^\perp&&}}
\eeq

Assume that $m = \lceil \frac{d-1}{2}\rceil$ is the minimal degree of a conic section on $\Phi_d$. As we know, in this case 
$\calC_S^{\norm}$ is isomorphic to $\bfF_1$ (if $d$ is even) or $\bfF_0$ (if $d$ is odd).
If $m$ is even, the composition of the map $\tilde{h}:\tilde{\Phi}_d\to \bfF_1$ with the blowing-down map 
$\bfF_1\to \bbP^2$ is given by the linear system 
$$|D| = |mH-(m-1)\Sigma-E|.$$
It is equal to the proper transform on $\tilde{\Phi}_d$ of the linear system $|mH-(m-1)\Gamma-S|$ of monoidal surfaces of degree $m$ containing $S$.
If $d$ is odd, this linear system is one-dimensional, and the map $\tilde{h}$ is given by the linear system
$$|D| = |mH-(m-1)\Sigma-E-(e_0-e_{3d-3})|.$$ 
In both cases, the map $\tilde{h}$ blows down the irreducible components of reducible fibers of the conic 
bundle which do not intersect $E$.

\begin{example} Let  $\Phi_3$ be a smooth cubic surface and $\Gamma$ be one of its lines. 
The sixteen lines skew to $\Gamma$ are the special sections of $\Phi_3$.  

Choose a special section $S$. The congruence of lines 
$\calC_S$ is the locus of lines intersecting $\Gamma$ and $S$. It is a congruence of lines 
embedded in the Pl\"ucker space as a quadric surface. Each ray of $\calC_S$ is the intersection of the planes 
$\Pi\in \Gamma^\perp$ and $\Pi'\in S^\perp$. The pencils $\Gamma^\perp$ and $S^\perp$ define the two rulings 
of $\calC_S$. The ruling $S^\perp$ corresponds to the monoidal surfaces $\Phi_1$. It is the ruling of directrices 
of the ruling $\Gamma^\perp$.
\end{example}

 \begin{example} Assume $d = 4$ and $\Phi_d^{\norm}$ is smooth. A special section $S$ from Corollary 
 \ref{cor:5.4} is a conic that intersects 
$\Gamma$ at one point $y_0$. By 
Corollary \eqref{cor:5.4}, there are 128 such conics (see Example \ref{quartics}). 

 The congruence of lines $\calC_S$ is a smooth 
ruled surface of degree $3$ in $G_1(\bbP^3)$. The pencil $\Omega(y_0,\la S\ra)$ of rays of $\calC_S$ 
  is a generator of $\calC_S$. Since each ray of $\calC_S$ intersects $\Gamma$ and a unique point on $S$,  
$\Omega(y_0,\la S\ra)$ is also a directrix of $\calC_S$ of degree one. The surface $\calC_S$ isomorphic to $\bfF_1$
embedded in $\bbP^5$ by the linear system $|2\frakf+\frake|$. It is contained in the special complex 
$\Omega(\Gamma)$.

\end{example}

 \section{Cremona involutions associated with submonoidal surfaces}
Recall that a pair $\{a,b\}$ of distinct points on $\bbP^1$ defines an involutions $\sigma_{a,b}$ uniquely 
determined by the property that $\sigma_{a,b}(a) = a$ and $\sigma_{a,b}(b) = b$. It sends a point $x$ to a unique point 
$x'$ such that $x+x', 2a, 2b$ belong to the same linear series 
$g_2^1$ of degree $2$ of $\bbP^1$. In classical terminology, the pair $\{x,x'\}$ is harmonically conjugate to the pair 
$\{a,b\}$. Another involution $\sigma_{a,b;p}$ associated with the pair 
$\{la a,b\}$ requires fixing one point $p$ on $\bbP^1$. It is uniquely determined by the property 
that $\sigma_{a,b;p}(a) = b, \sigma_{a,b}(b) = a, \sigma_{a,b}(p) = p$. 

 Any involution of $\bbP^1$ coincides 
with either $\sigma_{a,b}$ or $\sigma_{a,b;p}$ for some $a,b,p$. It is easy to check that the involutions 
$\sigma_{a,b}$ and $\sigma_{a,b;p}$ commute.

A plane submonoidal curve $\calH_d$ of degree $d$ with a $(d-2)$-multiple point $x_0$ defines a 
Cremona transformations $\dJ_{\calH_d,x_0}$ (resp. $\dJ_{\calH_d,x_0}'$)
of $\bbP^2$ uniquely determined by the property that its restriction to a general line $\ell$ passing through $x_0$ 
coincides with the involution $\sigma_{a,b}$ (resp. $\sigma_{a,b;x_0}'$), where $a,b$ are the residual intersection points 
of $\ell$ with $\calH_d$. The involutions  are the \emph{de Jonqui\`eres 
birational involutions} of the plane (see \cite[7.3.6]{CAG}). The fundamental points 
of $\dJ_{\calH_d,x_0}$ are the intersection points $x_i$ of the first polar $P_{x_0}(\calH_d)$ with 
$\calH_d$ (and the point $x_0$ if $d = 2$). The number of them is equal to $2d-1 = 2g+3$, where $g = d-2$ is 
the geometric genus of $\calH_d$. The transformation $dJ_{\calH_d,x_0}'$ lifts to a biregular involution of the 
blow-up 
$\Bl_{x_0,\ldots,x_{2d-1}}(\bbP^2)$ of the 
fundamental points. It sends the exceptional curve over $x_0$ to the proper transform of the 
polar curve $P_{x_0}(\calH_d)$ and sends other exceptional curves to the proper transforms of the lines 
$\la x_0,x_i\ra$ tangent to $\calH_d$ at $x_i$. If $d\ge 4$, the curve $\calH_d$ is a hyperelliptic curve 
and the points $x_i,i\ne 0,$ are the Weierstrass points of $\calH_d$. The set of fixed points of the lift of 
$\dJ_{\calH_d,x_0}'$ to $X$
is equal to the proper transform of the curve $\calH_d$. 

The transformation $\dJ_{\calH_d,x_0}$ lifts to a biregular involution of $\Bl_{x_0}(\bbP^2)$. Its fixed points
are the pre-images of the points $x_1,\ldots,x_{2d-2}$

In the special case when $d = 2$, the de Jonqui\`eres 
transformation $\dJ_{C,x_0}$ is the orthogonal transformation of the plane that fixes $x_0$ and fixes 
the polar line $P_{x_0}(C)$ pointwise.  More explicitly, let $\alpha$ be a vector of projective coordinates 
of $x_0$ and $C = V(q)$ for some quadratic form $q$, the transformation $\dJ_{C,x_0}$ is the reflection with respect to the 
line $P_{x_0}(C)$: 
\beq\label{reflection}
x = [v]\mapsto [q(\alpha)v-b_q(v,\alpha)\alpha].
\eeq
where $b_q$ is the symmetric bilinear form associated with $q$.

The de Jonqui\`eres transformation $\dJ_{C,x_0}'$ is the 
the standard Cremona involution with three fundamental points $x_0,x_1,x_2$, where 
$x_1,x_2$ are the intersection points of the polar line $P_{x_0}(C)$ with $C$.  

\begin{thm} Assume $\tilde{\Phi}_d = \Phi_d^{\bp}$. Fix an isomorphism $\tilde{\Phi}_{d}\cong \Bl_{\calP}(\bbP^2)$. Assume that $d = 2m$ and 
$|\Sigma+(m-1)(e_0-e_{3d-3})|$ contains a smooth curve whose image in $\bbP^2$ is a curve $\calH_{3m-2}$ of degree $3m-2$ 
containing $p_0$ with multiplicity $3m-4$ and tangent to each line $\la p_0,p_i\ra, i = 1,\ldots,3d-4,$ 
at 
the point $p_i$. Then $\tilde{\Phi}$ admits 
an involution 
$\tau_{\calP}$ that descends to a projective involution $\tau$ of $\Phi_{d}$ that sends a special section $S$ to the dual special section 
$S^{\dag}$. The involution $\tau$ does not depend on a choice of the blowing down map 
$\tilde{\Phi}_d\to \bbP^2$.
\end{thm}

\begin{proof} The points $p_i$ are the intersection points of $\calH_{3m-2}$ with its polar curve $
P_{p_0}(\calH_{3m-1})$. Let $\tau_\calP$ be the de Jonqui\`eres involution $\dJ_{\calH_d,p_0}'$ associated with the curve 
$\calH_{3m-1}$ which fixes the curve pointwise. 
The homaloidal linear system defining $\tau_\calP$ consists of curves 
$C^{3m-1}(p_0^{3m-2},p_1,\ldots,p_{3d-4})$. The exceptional curves of $\tau_\calP$
are the lines $\la p_0,q_i\ra$, and the first polar curve $P_{p_0}(\calH_d)$. The latter is the unique monoidal curve 
$C^{3m-2}(p_0^{3m-3},p_1,\ldots,p_{3d-4})$ that is tangent to the lines $\la p_0,p_i\ra$ at the points $p_i$. 
The exceptional lines 
$\la p_0,p_i\ra$ are blown down to the points $p_i$, and the first polar is blown down to $p_0$.
The birational transformation lifts to a biregular involution $\tau_\calP$. It sends exceptional curve 
$E_0$ over $p_0$ to the exceptional curve $E'$ equal to the proper transform of $P_{x_0}(\calH_d)$. 
It swaps the irreducible components of fibers of the conic bundle $\pi:\tilde{\Phi}_d\to \bbP^1$. 
it follows that it acts on the geometric basis 
$(e_0,e_1,\ldots,e_{3d-3})$ corresponding to the blow-up by 
\begin{eqnarray}\label{involution1}
e_0&\mapsto& (3m-1)e_0-\sum_{i=1}^{3d-4}e_i-(3m-2)e_{3d-3},\\ \notag
e_i&\mapsto& e_0-e_i-e_{3d-3},\\ \notag
e_{3d-3}&\mapsto&(3m-2)e_0-\sum_{i=1}^{3d-4}e_i-(3m-3)e_{3d-3}.
\end{eqnarray}
Substituting this in \eqref{specialsection}, we obtain that $E$ is mapped to $E^{\dag}$. This implies that 
the birational transformation of $\Phi_d$ induced by $\tau_\calP$ sends $S$ to $S^\dag$. 

Using \eqref{involution1}, we check that $\tau_\calP$ sends $H \sim 2m e_0-\sum_{i=1}^{3d-4}-(2m-2)e_{3d-3}$ to itself.
Thus $\tau_\calP$ descends to a projective transformation of $\Phi_d$. Since it leaves invariant $2^{6m-4}$ pairs 
$S+S^{\dag}$ cut out by monoidal surfaces of degree $m-1$ with $\Gamma$ as its $(m-3)$-multiple line, 
we see that the involution does not depend on $\calP$.
\end{proof}

\begin{remark} Assume $m\ge 2$. The plane curve $\calH_{3m-2}$ is the image of a smooth hyperelliptic curve of genus 
$g = 3m-3$ 
under a map given by the linear system 
$|g_1^2+q_1+\cdots+q_{3m-3}|$, where $q_i$ correspond to the branches of of $\calH_d$ at $p_0$. The moduli space of 
point sets $\{p_0,p_1,\ldots,p_{6m-4}\}$ realized as the singular point of $\calH_{d}$ and its Weierstrass points 
is of dimension $2g-1+3m-3 = 9m-10$. It is of codimension $12m-14 -(9m-10) = 3m-4$ in the moduli space of $6m-3$ points in the plane.
For example, the moduli space of quartic submonodidal surfaces contains a subvariety of codimension two 
parameterizing surfaces for which there exists a birational involution which swaps special conic sections with their duals.
\end{remark}

Let $\calS_{\Gamma}(\Phi_d)$ be the closure of the poles of $\Gamma$ with respect to smooth residual conics $C_\Pi$ in the planes 
$\Pi\in \Gamma^\perp$. We call it the \emph{satellite polar curve} of $\Phi_d$ with respect to $\Gamma$.

We define a birational transformation $\Theta_\Gamma$ of $\bbP^3$ uniquely determined by the property 
that its restriction to any plane $\Pi\in \Gamma^\perp$ coincides with the 
the de Jonqui\`eres transformation defined by reflection with respect to the line 
$\Gamma\subset \Pi$. It is the transformation $\dJ_{C,s}$, where $C$ is the residual conic 
in the plane and $s = \Pi\cap \calS_{\Gamma}(\Phi_d)$. Its restriction to any line $\ell$ in $\Pi\in \Gamma^\perp$ that passes through 
$s$ coincides with the involution $\sigma_{a,b;s} = \sigma_{a,b;g}$, where 
$g = \ell\cap \Gamma$m and $a,b$ are the residual intersection points of $\ell$ with $\Phi_d$. The transformation $\Theta_\Gamma$ fixes $\calS_{\Gamma}(\Phi_d)$ and 
$\Gamma$ pointwise.

\begin{prop} Fix a blow-up model $\tilde{\Phi}_d\cong \Bl_\calP(\bbP^2)$ of $\Phi_d$. Assume that $\Sigma$ is 
irreducible. The restriction $\theta_\Gamma$ of the birational involution $\Theta_\Gamma$ to $\Phi_d$ is defined by 
the de Jonqui\`eres involution associated with the curve $\Sigma$ and its $(d-3)$-multiple point $p_0$.
\end{prop}

\begin{proof}  Let $\phi:\bbP^2\da \Phi_d$ be the rational 
parameterization of $\Phi_d$ defined by the blow-up construction. The image of the line 
$\ell_p = \la p,p_0\ra$ in $\Phi_d$ is the residual conic $C_\Pi$ in the plane $\Pi = \la \phi(p),\Gamma\ra$. 
The image of $\Sigma$ is the line $\Gamma$. The restriction of $\Theta_\Gamma$ to the line $\ell_p$ is an involution with 
two fixed points from $\Sigma\cap \ell_p$. This shows that $\Theta_\Gamma$ defines an birational involution of the conic 
$C_\Pi$ with fixed point from $\Gamma\cap C_\Pi$. This coincides with the restriction of the 
reflection involution to the conic.
\end{proof}
Let us find an explicit parametric equation of $\calS_\Gamma(\Phi_d)$.

\vskip3pt

 We use equation \eqref{eq:Phi} of $\Phi_d$. Its intersection with a general plane 
 $\Pi_{[t_0,t_1]} = V(x_1t_0-x_0t_1)\in \Gamma^\perp$ is equal to the pole of the line 
 $\Pi_{[t_0,t_1]}= V(y_0)$ with respect to the conic $C_{[t_0,t_1]}$ given by \eqref{eq:conicbundle}. It is equal to the intersection point 
 of the polar lines of the conic with respect to the points $[0,0,1]$ and $[0,1,0]$.
 \begin{eqnarray}\label{polar}
 &B(t_0,t_1)y_0+D(t_0,t_1)y_1+E(t_0,t_1)y_2 = 0,\\  \notag
 &C(t_0,t_1)y_0+E(t_0,t_1)y_1+F(t_0,t_1)y_2 = 0.
  \end{eqnarray}
 The intersection point has coordinates 
 $$[y_0,y_1,y_2] = [\Delta_1(t_0,t_1),\Delta_2(t_0,t_1),\Delta_3(t_0,t_1)],$$
 where $\Delta_i$ are the signed maximal minors of the matrix of the coefficients of this system of linear equations 
 in $y_0,y_1,y_2$.
 
 The image of the corresponding point $([t_0,t_1],[y_0,y_1,y_2])\in \bbP$ to $\bbP^3$ is equal to 
 \beq\label{satellite}
[y_0t_0,y_0t_1,y_1,y_2] = [t_0\Delta_1,t_1\Delta_1,\Delta_2,\Delta_3].
\eeq
 After cancelling by a common divisor, we get a rational parameterization of the satellite polar curve 
 $\calS_\Gamma(\Phi_d)$.

 It follows from \eqref{discriminant} that the curve $\calS_\Gamma(\Phi_d)$ passes through
 the singular points of the residual conics. It also passes through the points where they intersect 
 $\Gamma$ with multiplicity $2$. 
 
 We expect that, in general, $\calS_\Gamma(\Phi_d)$ is a curve of degree $2d-3$ that 
contains $3d-4$ singular points of residual conics 
and $2d-4$ points in the small discriminant $\calD_2$ in $\Gamma$. 
  
 \begin{example} The polar curve $\calP_\Gamma(\Phi_d)$ could be reducible. This happens when the binary forms 
 in the parametric equation have common factors. For example, consider a cubic surface
 $$2x_0^2x_2+2x_1^2x_3+x_0x_2^2+x_1x_3^2 = 0.$$
  The surface
has two Eckardt points $[0,0,1,0]$ and $[0,0,1,0]$ on $\Gamma = V(x_0,x_1)$. 
The pencil of polar quadrics is given by 
$$t_0 (x_0^2+x_0x_2)+t_1 (x_1^2+x_1x_3) = 0.$$
Its base locus is the union of four lines $V(x_0(x_0+x_2),x_1(x_1+x_3))$. The residual part consists of 
three lines $V(x_0,x_3), V(x_1,x_2)$ and $V(x_0+x_2,x_1+x_3)$. The latter line is given by parametric equation 
$[t_0,t_1,t_0,t_1]$. The first two of the three lines are irreducible components of the 
intersection of the surface with the two tritangent planes. The third line coincides with 
$\calS_\Gamma(\Phi_3)$. It  does not lie on the surface and joins 
the singular points of three remaining reducible residual conics. 
 \end{example}

 \vskip5pt
\begin{prop} The involution $\Theta_\Gamma$ is given by the formula
\begin{eqnarray*}
x_0&\mapsto& -x_0\Delta_1(x_0,x_1),\\
x_1&\mapsto& -x_1\Delta_1(x_0,x_1),\\
x_2&\mapsto& x_2\Delta_1(x_0,x_1)-2\Delta_2(x_0,x_1),\\
x_3&\mapsto&  x_3\Delta_1(x_0,x_1)-2\Delta_3(x_0,x_1).
\end{eqnarray*}
\end{prop}

\begin{proof} The restriction of $\Theta_\Gamma$ to a general plane $\Pi_{[t_0,t_1]}$ is equal to the  
the reflection involution 
$$y = [y^*] = [y_0,y_1,y_2] \mapsto [q_t(p_0^*)y^*-2b_{t}(y,p_0^*)p_0^*],$$
where $p_0 = [p_0^*]$ is the pole of the line $V(y_0)$ with respect to the conic $C_{[t_0,t_1]} = V(q_t)$, 
and $b_t(v,w) = \half (q_t(v+w)-q(v)-q(w))$ is the polar symmetric bilinear form associated with $q$. 
Substituting the coordinates $p_0^* = (\Delta_1,\Delta_2,\Delta_3)$ in the equation of the 
conic \eqref{eq:conicbundle}, we find that the restriction of $\Theta_\Gamma$ to the plane is given by
$$[y_0,y_1,y_2]\mapsto (-y_0\Delta_1,y_1\Delta_1-2y_0\Delta_2,y_2\Delta_1-2y_0\Delta_3].$$
Using the translation of coordinates in $\bbP$ and in $\bbP^3$, we get the asserted formula.
\end{proof}

By inspection of the formula, we find that the algebraic degree of $\Theta_\Gamma$ for a general $\Phi_d$ 
is equal to 
$2d-3$. Its indeterminacy locus is equal to the union of $\Gamma$ and $3d-4$ singular points of the residual conics.
Its exceptional divisor is equal to the union of the  planes containing $2d-4$ residual 
conics that are tangent to $\Gamma$. They are blown down to the line $\Gamma$. The fixed locus of 
$\Theta_\Gamma$ is equal to the satellite polar curve.

\begin{prop}  The involution $\Theta_\Gamma'$ is given by the formula
\begin{eqnarray*}
x_0&\mapsto& x_0(F_d(x_0,x_1,x_2,x_3)\Delta_1(x_0,x_1)-P(x_0,x_1)),\\
x_1&\mapsto& x_1(F_d(x_0,x_1,x_2,x_3)\Delta_1(x_0,x_1)-P(x_0,x_1)),\\
x_2&\mapsto& F_d(x_0,x_1,x_2,x_3)\Delta_2(x_0,x_1)-P(x_0,x_1)x_2,\\
x_3&\mapsto& F_d(x_0,x_1,x_2,x_3)\Delta_3(x_0,x_1)-P(x_0,x_1)x_3.
\end{eqnarray*}
where $\Phi_d = V(F_d)$.
\end{prop}

 \begin{proof} We use the notation from the proof of the previous proposition. The restriction of $\Theta_\Gamma$ to to a general plane $\Pi_{[t_0,t_1]}$ is equal to the  
the de Jonqui\`eres involution that fixes the conic $C_{[t_0,t_1]}$. 
 For a general point $y$ in the plane,
the involution sends the point $y$ to the unique point $y'$ on the line $\ell_y = \la y,p_0\ra$ such that the pair of points 
$\{y,y'\}$ is harmonically conjugate to the pair $\{a,b\} = \ell_y\cap C_{[t_0,t_1]}$.
Recall that, if we fix projective coordinates $(u,v)$ on the line $\ell_y$ such that 
$\{a,b\} = V(\alpha u^2+2\beta uv+\gamma v^2)$ and 
$\{y,y'\} = V(\alpha' u^2+2\beta' uv+\gamma' v^2)$, then the condition of harmonic conjugacy is 
$\alpha\gamma'-2\beta\beta'+\gamma\alpha' = 0$ \cite[2.1.2]{CAG}. This allows us to find the explicit formula.
The line
$\ell_y:[uy^*+vp_0^*]$
intersects the conic $C_{[t_0,t_1]}= V(q_t)$ at the points 
\beq
\{a,b\} = V(u^2q_t(y^*)+2uvb_t(y^*,p_0^*)+v^2q_t(p_0^*))
\eeq
Since $y$ has coordinates $[u,v] = [1,0]$, 
$$\{y,y'\} = V(v(2\beta'u+\gamma'v)).$$ 
Hence $y'$ is determined by the condition 
$q_t(y^*)\gamma'-2b_t(y^*,p_0^*)\beta' = 0$. The coordinates $[u,v]$ of $y'$ are equal to 
$[\gamma',-2\beta'] = [b_t(y^*,p_0^*),-q_t(y^*)]$. This gives 
$$y' = b_t(y^*,p_0^*)y^*-q_t(y^*)p_0^*.$$
Let $A$ be the matrix from \eqref{discriminant}. We have $b_t(y^*,p_0^*) = y^*\cdot A\cdot p_0^* = y_0|A| = 
y_0P(t_0,t_1)$.  This gives
{\Small 
 $$y'= [y_0^2P(t_0,t_1)-q_t(y^*)\Delta_1(t_0,t_1),y_0y_1P(t_0,t_1)-q_t(y^*)\Delta_2(t_0,t_1),
y_0y_2P(t_0,t_1)-q_t(y^*)\Delta_3(t_0,t_1)]
$$}
where $P(t_0,t_1)$ is the discriminant of the conic bundle. Going back to the coordinates in $\bbP^3$, we get the asserted formula.
 
 \end{proof} 
 
 By inspection of the formulas, we find that the algebraic degree of $\Theta_\Gamma'$ for general $\Phi_d$ is equal 
 to the union of $\Gamma$, the satellite polar curve $\calS_\Gamma(\Phi_d)$, and the singular 
 residual conics.
 Its exceptional divisor consists of the union of planes spanned by singular conics and the planes that contain an irreducible component 
 of a singular conic and tangent to $\calS_\Gamma(\Phi_d)$ at the singular point of the conic. The degree of the jacobian is equal 
 to $4(3d-4)$, and it is equal to $2(3d-4)+(3d-4)+(3d-4)$. The closure of the locus of fixed points on the domain of the definition 
 of  $\Theta_\Gamma'$ is equal to 
 $\Phi_d$.

 \begin{remark} Instead of taking the satellite polar curve $\calS_\Gamma(\Phi_d)$ intersecting a general plane 
$\Pi\in \Gamma^\perp$ at one point,
one can take a special section $S$ of the conic bundle. The Cremona involution 
$\sigma_{\Phi_d,S}$ is uniquely defined by the property that its restriction to 
any line contained in $\Pi\in \Gamma^\perp$ and passing through the point $S\cap  \calS_\Gamma(\Phi_d)$ 
coincides with the involution $\sigma_{a,b}$ that fixes the intersection points of the line 
with $\Phi_d$ different from $s$ and a point on $\Gamma$. Similarly, we define the involution 
$\sigma_{\Phi_d,S}'$ that swaps the points $a$ and $b$.  

For example, a choice of a line on a smooth cubic surface $\Phi$ defines sixteen Cremona involutions
$\sigma_{\Phi,S}$, and sixteen Cremona involutions
$\sigma_{\Phi,S}'$. 
\end{remark}

\begin{remark} We do not know whether the natural homomorphism 
$$\textrm{Bir}(\bbP^3,\Phi_d)\to \textrm{Bir}(\Phi_d)$$ 
is surjective, nor do we know its kernel. The surjectivity is known 
in the case $d\le 4$ because the surfaces are Cremona equivalent to a plane \cite{Mella}.  
Even if the map is surjective, we do not know whether any birational involution lifts to a birational involution of $\bbP^3$.
\end{remark}

 \section{The satellite polar variety} 
 The construction of the Cremona involutions and the satellite polar curve $\calS_\Gamma(\Phi_d)$ 
 can be extended to a submonoidal hypersurface $\Phi_d^n$ in $\bbP^{n+1}$ which contains
 a linear subspace $\Gamma$ of codimension $m > 1$ with multiplicity $d-2$. It is  
given by equation
{\small $$A_d(x_0,\ldots,x_{m-1})+2\sum_{i=m}^{n+1}l_i(x_0,\ldots,x_{m-1})x_{i}+
\sum_{i,j =m}^{n+1}q_{ij}(x_0,\ldots,x_{m-1})x_{i}x_{j}= 0,
$$}
where $l_i$ (resp. $q_{ij} = q_{ji}$) are homogeneous polynomials of degree $d-1$ (resp. $d-2$).
The residual quadrics are cut out by the linear system $\Gamma^\perp$ of dimension $m$  of $n-m+1$-dimensional 
linear subspaces containing $\Gamma$. 

As in the case of surfaces, we can define the \emph{polar variety} $\calP_\Gamma(\Phi_d^n)$ to be the residual part 
of  $(d-2)\Gamma$ in $\cap_{x\in \Gamma}P_x(\Phi_d^n)$. Its irreducible 
component $\calS_\Gamma(\Phi_d)$ equal to the closure of singular points of quadrics $Q_t\subset \Pi_t, t\in \Gamma^\perp,$ of corank 1 is called the 
\emph{satellite polar variety}.

Let us find its parametric equation. First we write $\Pi_t$ as the span of $\Gamma$ and a point
$[t_0,\ldots,t_{m-1},0,\ldots,0]$. 
This gives the coordinates $[y,x_{m},\ldots,x_{n+1}]$ in $\Pi_t$ such that any point in  $\Pi_t$  has coordinates $
[t_0y,\ldots,t_{m-1}y,x_{m},\ldots,x_{n+1}]$ in $\bbP^{n+1}$. Substituting this in equation of $\Phi_d^n$, we obtain
the equation of the residual quadric $Q_t$
\beq
{\small A_d(t_0,\ldots,t_{m-1})y^2+2y\sum_{i=m}^{n+1}l_i(t_0,\ldots,t_{m-1})x_{i}+
\sum_{i,j=m}^{n+1}q_{ij}(t_0,\ldots,t_{m-1})x_{i}x_{j} = 0.}
\eeq
The equation of $\Gamma$ in $\Pi_t$ is $y = 0$. 
Let $[y_0,\alpha_m,\ldots,\alpha_{n+1}]$ be the coordinates of the pole $P_t$ of $Q_t$ 
with respect to the hyperplane $\Gamma\subset \Pi_t$. Then polar hyperplane of $P_t$ with respect to the quadric $Q_t$
is the hyperplane $\Pi_t = V(y)$. This gives a system of linear equatih unknowns $y_0,\alpha_m,\ldots,\alpha_{n+1}$ 
\beq
A= \begin{pmatrix}l_m(t_0,\ldots,t_{m-1})&q_{mm}(t_0,\ldots,t_{m-1})&\ldots&q_{mn+1}(t_0,\ldots,t_{m-1})\\
\vdots&\vdots&\vdots&\vdots\\
l_{n+1}(t_0,\ldots,t_{m-1})&q_{n+1m}(t_0,\ldots,t_{m-1})&\ldots&q_{n+1n+1}(t_0,\ldots,t_{m-1}\end{pmatrix}
\eeq
defined by the partial derivatives of $Q_t$ with respect to the variables $y,x_{m},\ldots,x_{n+1}$.

This gives us a rational parameterization of $\calS_\Gamma(\Phi_d)$
\beq
\mathfrak{r}: \bbP^{m-1}\da \calS_\Gamma(\Phi_d), 
\eeq
given by 
\beq  [x_0,\cdots,x_{n+1}] = [t_0\Delta_{1},\cdots,t_{m-1}\Delta_{1},-\Delta_{2},\cdots,(-1)^{n+2-m}\Delta_{n+3-m}],
\eeq 
where $\Delta_{i}$ is the maximal minor of the matrix $A$ obtained by deleting the $i$th column. 
It is of degree equal to $(n+2-m)(d-2)+1$.

 Since the codimension of the space of matrices of size
$(n-m+2)\times (n-m+3)$ and corank $\ge 1$ is equal to $2$, the subvariety of $\Gamma^\perp$ given by vanishing of the 
minors $\Delta_m.\ldots,\Delta_{n+1}$ is expected to be of codimension $2$ in $\bbP^{m-1}$. 
So, if $m\ge 3$, the map $\mathfrak{r}$ is not a regular map.

 As in the case of submonoidal surfaces, we can define two commuting birational transformations of 
$\bbP^{n+1}$ that leaves $\Phi_d^n$ invariant and restricts to the identity on  
$\calS_\Gamma(\Phi_d^n)$.

Let  $\Theta_\Gamma$ (resp. $\Theta_\Gamma'$) be the birational involution of $\bbP^{n+1}$ uniquely defined by the property that its restriction 
to any line in $\Pi_t$ that passes through the point  $s = \Pi_t\cap \calS_\Gamma(\Phi_d^n)$ is the involution 
that interchanges (resp. fixes) the intersection points of the line with the residual quadric $Q_t$.j

The restriction of $\Theta_\Gamma$ (resp. $\Theta_\Gamma'$) to $\Phi_d$ is a non-trivial 
birational involution $\theta_\Gamma$ of $\Phi_d^n$
(resp. the identity).

\begin{example} In the case $m = n+1$, i.e. $\Gamma$ is a point, the satellite polar variety $\calS_\Gamma(\Phi_d^n)$ coincides with 
the monoidal polar hypersurface $P_\Gamma(\Phi_d^n)$ of degree $d-1$. The residual conics are just pairs of points, the 
fibers of the projection $\Phi_d^{n}$ from the point $\Gamma$. The polar hypersurface passes through all ramification points.
 A general line in $\bbP^{n+1}$
passing through $\Gamma$ intersects $\calS_\Gamma(\Phi_d^n)$ at one point $s$ different from $\Gamma$, 
the ramification point of the projection 

The involution $\Theta_\Gamma$ is defined by the projection that its projection to any 
line $\ell$ passing through $\Gamma$ coincides with the involution $\sigma_{a,b,s} = \sigma_{a,b;\Gamma}$, where 
$s$ is the residual point of the intersection of $\ell$ with $\Phi_d^n$ and $a,b$ are the residual intersection points of the line and $\Phi_d^n$. Its restriction to $\Phi_d^n$ is defined by the 
projection from the point $\Gamma$.

The involution $\Theta_\Gamma'$ is a de Jonqui\`eres 
transformation associated to a submonoidal hypersurface. Its restriction to 
$\ell$ coincides with the involution $\sigma_{a,b}'$. Both involutions were studied by M. Gizatullin in 
\cite{Gizatullin}. He gives the explicit formulas for these involutions as well as some relations between them.  
J. Blanc studied dynamical properties of the compositions of involutions $\Theta_\Gamma$ on a cubic hypersurface \cite{Blanc}.

\end{example}

\begin{example} Assume that $\Phi_d^n$ is a general $n$-dimensional cubic hypersurface $\Phi_3^n$ in $\bbP^{n+1}$ 
with a fixed linear subspace $\Gamma\subset \bbP^{n+1}$ of codimension $m$. It is known that the dimension 
of the variety of linear subspaces $\Gamma\subset \Phi_d^{n}$ of codimension $m$ in $\bbP^{n+1}$ is greater or 
equal 
than $(n+2-m)m-\binom{n-m+d+1}{d}$ \cite[Lecture 12]{Harris}. So, we can always find a subspace $\Gamma$ of codimension $m$ 
in $\Phi_3^n$ if $m$ satisfies $(n+2-m)m-\binom{n-m+4+d}{3}\ge 0$. Of course, we can always take $m= n+1$ or $n$, since a cubic hypersurface 
of any dimension always contains a line. The satellite variety $\calS_\Gamma(\Phi_3)$ is a rational variety 
of dimension $m-1$ given by a rational parameterization of degree $n-m+2$.

Following \cite{Gizatullin}, one defines the decomposition (resp. the inertia) group of $\Phi_d^n$ as the subgroup 
of the Cremona group $\Cr(n+1)$ that leaves $\Phi_d^n$ invariant (resp. restricts to the identity on $\Phi_d^n$).

It is known that in the case of surfaces, the group $\textrm{Bir}(\Phi_3)$ is generated by the transformations $\theta_\Gamma$, where 
$\Gamma$ is a point \cite{Manin}. Is it true for $d>3$? This would imply that the decomposition group is mapped surjectively onto 
$\textrm{Bir}(\Phi_d)$. I do not know whether 
the same is true in dimension greater than $2$, i.e. whether $\textrm{Bir}(\Phi_d^n)$ is generated by transformations 
$\theta_\Gamma$, and whether it is enough to take $\Gamma$ to be a point. Also, I do not know whether the inertial group 
of $\Phi_d^n$ is generated by transformations $\Theta_\Gamma'$.

\end{example}


\end{document}